\documentclass[a4paper]{article}

\pdfoutput=1

\usepackage{amsmath,amssymb,amsthm}
\usepackage{mathtools}

\usepackage{xcolor}

\usepackage{tikz}
\usetikzlibrary{shapes, matrix, arrows}
\usetikzlibrary{decorations.pathmorphing}
\usetikzlibrary{intersections}
\usetikzlibrary{patterns}

\usepackage{hyperref}

\newcommand{\union}{\cup}
\newcommand{\bigunion}{\bigcup}
\newcommand{\disjointunion}{\sqcup}
\newcommand{\freeprod}{\ast}
\newcommand{\intersection}{\cap}
\newcommand{\boundary}{\partial}

\DeclareMathOperator{\Ball}{\mathrm{Ball}}
\newcommand{\isom}{\cong}
\DeclareMathOperator{\shadow}{\mathcal{S}}
\DeclareMathOperator{\Cay}{\mathrm{Cay}}
\DeclareMathOperator{\val}{\mathrm{val}}
\DeclareMathOperator{\im}{\mathrm{Im}}

\DeclareMathOperator{\inc}{\mathrm{Inc}}

\DeclareMathOperator{\Isom}{\mathrm{Isom}}
\DeclareMathOperator{\composed}{\circ}

\newcommand{\gromprod}[3][1]{\ensuremath{\left(#2\cdot#3\right)_{#1}}}
\renewcommand{\mod}[1]{\vert#1\vert}
\newcommand{\restricted}[1]{\vert_{#1}}
\newcommand{\integers}{\mathbb{Z}}
\newcommand{\naturals}{\mathbb{N}}
\newcommand{\reals}{\mathbb{R}}

\newcounter{dummy}\numberwithin{dummy}{section}
\newtheorem{lem}[dummy]{Lemma}
\newtheorem{cor}[dummy]{Corollary}
\newtheorem{thm}[dummy]{Theorem}
\newtheorem{rem}[dummy]{Remark}
\newtheorem{prop}[dummy]{Proposition}
\newtheorem{defn}[dummy]{Definition}

\title{Computing JSJ decompositions of hyperbolic groups}
\author{Benjamin Barrett}

\begin{document}
\maketitle

\bibliographystyle{hplain}

\begin{abstract} We present an algorithm that computes Bowditch's canonical JSJ
decomposition of a given one-ended hyperbolic group over its virtually cyclic
subgroups. The algorithm works by identifying topological features in the boundary
of the group. As a corollary we also show how to compute the JSJ
decomposition of such a group over its virtually cyclic subgroups with
infinite centre. We also give a new algorithm that determines whether or not a
given one-ended hyperbolic group is virtually fuchsian. Our approach uses only
the geometry of large balls in the Cayley graph and avoids Makanin's
algorithm.\end{abstract}

\section{Introduction}

When studying a group it is natural and often useful to try to cut it into
simpler pieces by means of amalgamated free products and HNN extensions over
particularly simple subgroups. Sometimes this can be done in a canonical way
analogous to the characteristic submanifold decomposition of Jaco, Shalen and
Johannson~\cite{jacoshalen79, johannson79}, in which the family of embedded
tori along which the 3-manifold is cut is unique up to isotopy. Such JSJ
decompositions were introduced to group theory by Sela~\cite{sela97} to answer
questions about rigidity and the isomorphism problem for torsion-free
hyperbolic groups. In~\cite{bowditch98} Bowditch developed a related type of
decomposition for hyperbolic groups possibly with torsion. This decomposition
is built from the structure of local cut points in the boundary of the group
and is therefore an automorphism invariant of the group; this property of the
Bowditch JSJ was used in Levitt's work~\cite{levitt05} on outer automorphism
groups of one-ended hyperbolic groups. For more general constructions of JSJ
decompositions of groups see~\cite{ripssela97, dunwoodysageev99,
fujiwarapapsoglu06, guirardellevitt16}.

The above results describe and prove the existence of various types of JSJ
decompositions but do not give an algorithm to construct them.
Gerasimov~\cite{gerasimov} proved that there exists an algorithm that
determines whether or not the Gromov boundary of a given hyperbolic group is
connected. This algorithm is unpublished; see also~\cite{dahmanigroves08a}. The
connectedness of the boundary is determined by the so-called double-dagger
condition of Bestvina and Mess~\cite{bestvinamess91}; it is this condition that
Gerasimov showed to be computable. Equipped with this algorithm and Stallings's
theorem on ends of groups it is not difficult to compute a maximal decomposition of
a given hyperbolic group over its finite subgroups. With Gerasimov's result in
hand, we may restrict to the case of one-ended hyperbolic groups and consider
the computability of Bowditch's JSJ decomposition over virtually cyclic
subgroups. 

In this paper we present an algorithm that computes Bowditch's decomposition.
Like Gerasimov's algorithm, our approach uses the geometry of large balls in
the Cayley graph. This is in contrast to existing algorithms computing JSJ
decompositions over restricted families of virtually cyclic subgroups, most of
which rely on Makanin's algorithm for solving equations in free groups.

In~\cite{dahmaniguirardel11} Dahmani and Guirardel show that a canonical
decomposition of a one-ended hyperbolic group over a particular family of
virtually cyclic subgroups is computable; the family in question is the set of
virtually cyclic subgroups with infinite centre that are maximal for inclusion
among such subgroups. Crucial to this method is an algorithm that determines
whether or not the outer automorphism group of such a group is infinite. If a
group admits such a splitting then that splitting gives rise to an infinite set
of distinct elements of the outer automorphism group that are analogous to Dehn
twists in the mapping class group of a surface. The converse of this statement
is a theorem of Paulin~\cite{paulin91} that is refined by Dahmani and
Guirardel.

Dahmani and Guirardel comment that it is not known whether or not Bowditch's JSJ
decomposition is computable. Their approach is not suitable to this problem:
only central elements of the edge groups in a splitting contribute Dehn twists
to the automorphism group, so it is quite possible for a group to admit a
splitting over an infinite dihedral group, say, while having only a finite
outer automorphism group; in this case the decomposition computed by Dahmani
and Guirardel is trivial while Bowditch's JSJ decomposition is not. For
examples of hyperbolic groups exhibiting this property
see~\cite{millerneumannswarup96}. 

In the absence of torsion, the JSJ decomposition of a hyperbolic group over its
cyclic subgroups was shown to be computable by Dahmani and Touikan
in~\cite{dahmanitouikan13}. Their result is based on Touikan's
algorithm~\cite{touikan09}, which determines whether or not a given one-ended
hyperbolic group without 2-torsion splits acylindrically. Touikan's methods are
based on application of the Rips machine.

The existence of a splitting of a one-ended hyperbolic group over a virtually
cyclic subgroup is reflected in the existence of certain topological features
in its Gromov boundary by results of Bowditch~\cite{bowditch98, bowditch99a,
bowditch99b}; in this paper we show that these topological features can be
detected algorithmically.

\begin{thm}\label{thm:maintheorem} There is an algorithm that takes as input a
  presentation for a one-ended hyperbolic group and returns the graph of groups
  associated to the three following JSJ decompositions:
  \begin{enumerate}
    \item A JSJ decomposition over virtually cyclic subgroups of $\Gamma$,
      which we shall call a $\mathcal{VC}$-JSJ. This decomposition can be
      taken to be Bowditch's canonical decomposition.
    \item A JSJ decomposition over virtually cyclic subgroups of $\Gamma$
      with infinite centre, which we shall call a $\mathcal{Z}$-JSJ.
    \item A decomposition over maximal virtually cyclic subgroups of $\Gamma$
      with infinite centre that is universally elliptic over (not necessarily
      maximal) virtually cyclic subgroups of $\Gamma$ and is maximal for
      domination in the class of such decompositions. We shall call this a
      $\mathcal{Z}_\text{max}$-JSJ.
  \end{enumerate}
\end{thm}

The main content of Theorem~\ref{thm:maintheorem} is the computability of the
$\mathcal{VC}$-JSJ; it is shown in~\cite{dahmaniguirardel11} to be closely
related to the $\mathcal{Z}$-JSJ and $\mathcal{Z}_\text{max}$-JSJ and can
converted into either algorithmically. The $\mathcal{Z}_\text{max}$-JSJ is the
decomposition shown to be computable by Dahmani and
Guirardel~\cite{dahmaniguirardel11}.

The $\mathcal{Z}$-JSJ also plays a result in Dahmani and Guirardel's work,
although they comment that their methods cannot compute this decomposition,
since such a decomposition does not necessarily give rise to infinitely many
distinct outer automorphisms of the group.  For example, let $\Gamma$ be a
rigid hyperbolic group (such as the fundamental group of a closed hyperbolic
3-manifold) and let $g$ be an element of $\Gamma$ that is not a proper power.
Let $k > 1$ and consider the group $\Gamma' = \Gamma \freeprod_{g=t^k} \langle
t\rangle$ obtained by adjoining a $k$th root of $g$ to $\Gamma$. In this case
the $\mathcal{Z}_\text{max}$ decomposition computed by Dahmani and Guirardel is
trivial while the JSJ decomposition over virtually cyclic subgroups with
infinite centre is not. 

Central to the algorithm of Theorem~\ref{thm:maintheorem} is an algorithm that
determines whether or not a given hyperbolic group with a (possibly empty)
finite collection of virtually cyclic subgroups admits a proper splitting as an
amalgamated product or HNN extension over a virtually cyclic subgroup, relative
to that collection of subgroups. Recall that a \emph{cut pair} in a connected
topological space $S$ is a pair of points $p$ and $q$ such that $S - \{p, q\}$
is disconnected. It is shown in~\cite{bowditch98} that in the absolute case
(that is, if the collection of subgroups is empty), a one-ended hyperbolic
group admits such a splitting if and only if its Gromov boundary contains a cut
pair, at least as long as its boundary is not homeomorphic to a circle. In the
case of interest here we obtain a relative version of this statement by
replacing the Gromov boundary with the Bowditch boundary of the group relative
to the given family of subgroups. Unlike the Gromov boundary, the Bowditch
boundary might contain a cut point, in which case the group admits a relative
splitting. This is the peripheral splitting in the sense of
Bowditch~\cite{bowditch99b}. In the absence of a cut point the existence of a
relative splitting is determined by the existence of a cut pair, as in the
absolute case. It is the presence of these topological features of the boundary
that we show to be computable.

To detect the presence of a cut pair in the Bowditch boundary of a hyperbolic
group relative to a given family of subgroups  we first show that the
connectivity of the complement of a pair of points in the boundary is
equivalent to the connectivity of a thickened cylinder around a geodesic
connecting that pair of points in the cusped space defined
in~\cite{grovesmanning08}. Then, supposing that there is a cut pair in the
boundary, we use a pumping lemma argument to show that the geodesic $\gamma$
connecting the points in a cut pair may be assumed to be periodic and with
bounded period: we take a short subsegment $\gamma\restricted{[a, b]}$ of that
geodesic such that both the geodesic and the components of the thickened
cylinder are identical in small neighbourhoods of $a$ and $b$ and form a new
(local) geodesic that also connects the two points in a (possibly different)
cut pair by concatenating infinitely many copies of $\gamma\restricted{[a,
b]}$. A similar method is used in~\cite{cashenmacura11} to control cut pairs in
the decomposition space of a line pattern in a free group. This is sufficient
to detect a cut pair: the existence of such a periodic geodesic can be detected
in finite time by searching a large finite ball in the Cayley graph.

A maximal splitting is obtained from a JSJ decomposition by refining at the
flexible vertices. Conversely, to obtain a JSJ decomposition we must decide
which edges of the maximal splitting should be collapsed to reassemble the
flexible vertices in the JSJ decomposition. In Bowditch's JSJ decomposition the
stabilisers of flexible vertex groups are the maximal hanging fuchsian
subgroups. These are those subgroups that occur as a vertex stabiliser in some
splitting such that the Bowditch boundary of the subgroup relative the
stabilisers of the incident edges is homeomorphic to a circle. Therefore we
prove the following theorem, which is interesting in its own right, and does
not seem to appear in the literature.  Recall that the Convergence Group
Theorem of Tukia, Casson, Jungreis and Gabai~\cite{tukia88, cassonjungreis94,
gabai92} implies that the Gromov boundary of a hyperbolic group is homeomorphic
to a circle if and only if the group surjects with finite kernel onto the
fundamental group of a compact hyperbolic orbifold.

\begin{thm}\label{thm:S1boundarycomputable} There is an algorithm that takes as
input a hyperbolic group $\Gamma$ and a (possibly empty) collection of
virtually cyclic subgroups $\mathcal{H}$ and returns an answer to the
question ``is $\boundary(\Gamma, \mathcal{H})$ homeomorphic to
$S^1$?''\end{thm}

The algorithm of Theorem~\ref{thm:S1boundarycomputable} is similar to the
algorithm that detects the presence of a cut pair in the Bowditch boundary: it
follows from a result in point-set topology that the boundary is homeomorphic
to a circle if and only if every pair of points in the boundary is a cut pair.
We show that if there is a non-cut pair in the Bowditch boundary then there is
a non-cut pair connected in the cusped space by a local geodesic with bounded
period. 

In section~\ref{sec:cuspedspace} we first review the definition of the cusped
space and the Bowditch boundary. We then recall some important properties: the
computability of the hyperbolicity constant of the cusped space, the existence
of a visual metric on the Bowditch boundary, and most importantly the so-called
double-dagger condition, which will be vital in linking the connectivity of the
boundary to that of subsets of the cusped space. We then define a thickened
cylinder around a geodesic in the cusped space and show that its connectivity
determines the connectivity of the complement of the limit set of that geodesic
in the boundary.

Section~\ref{sec:cutpairsalgorithms} contains the main technical results of the
paper we describe the algorithms that determine whether or not the Bowditch
boundary of a hyperbolic group contains the three topological features of
interest to us: cut points, cut pairs and non-cut pairs.

In section~\ref{sec:fuchsiangroups} we deal with the special case of a group
with circular Bowditch boundary. We first recall some results that reduce the
problem of determining whether or not such a group admits a proper
splitting relative to its given virtually cyclic subgroups to the case in which the
group is the fundamental group of a compact two-dimensional hyperbolic orbifold
and the given subgroups are precisely conjugacy class representatives of the
fundamental groups of the boundary components of that orbifold. In
\cite{guirardellevitt16} a complete list of such orbifolds that do not admit
such a splitting is described; we use this to complete this special case.

In section~\ref{sec:maximal} we first record the general definition of a JSJ
decomposition and a description of Bowditch's canonical JSJ decomposition over
virtually cyclic subgroups. We then recall the theorem of~\cite{bowditch98}
that links the topology of the Gromov boundary of a hyperbolic group to the
existence of a proper splitting of that group and extend it to a relative
version. We then show how to use the algorithms described so far to compute a
maximal splitting of a one-ended hyperbolic group over virtually cyclic
subgroups.

In section~\ref{sec:JSJcomputable} we complete the proof of
Theorem~\ref{thm:maintheorem} by describing the processes that convert a
maximal splitting of a one-ended hyperbolic group over virtually cyclic
subgroups into its $\mathcal{VC}$-JSJ, $\mathcal{Z}$-JSJ and
$\mathcal{Z}_\text{max}$-JSJ.

It seems plausible that the techniques of this paper might extend to the
problem of detecting the presence of splittings of relatively hyperbolic groups
with parabolic subgroups in some restricted class. In particular, it is natural
to try to solve this problem for groups that are hyperbolic relative to
finitely generated virtually nilpotent subgroups. Such groups arise as
fundamental groups of complete finite volume Riemannian manifolds with pinched
negative sectional curvature. However, it is of fundamental importance to the
argument presented in this paper that the cusped space associated to the group
satisfies Bestvina and Mess's double-dagger condition, and we do not know under
what circumstances this condition holds for virtually nilpotent parabolic
subgroups. In~\cite{dahmanigroves08a} it is established that the double-dagger
condition holds if the parabolic subgroups are abelian, so it seems likely that
the methods of this paper could be extended to the case of toral relatively
hyperbolic groups. (A group is toral relatively hyperbolic if it is torsion
free and hyperbolic relative to a finite collection of abelian subgroups.)
However, the JSJ decomposition of a toral relatively hyperbolic group is shown
to be computable in~\cite{dahmanitouikan13}, so we do not introduce additional
technical complexity by trying to give a new proof of this result here.

\subsection*{Acknowledgements}

I thank my supervisor Henry Wilton for suggesting this problem and for his
guidance towards its solution.  I am also grateful to the anonymous reviewer
for pointing out several problems with the clarity and accuracy of the earlier
draft.

\section{The cusped space and cut pairs} \label{sec:cuspedspace}

\subsection{The cusped space and the boundary} 

In this section we recall some technical results about the geometry of
relatively hyperbolic groups.  Fix a group $\Gamma$ with a finite generating
set $S$ and a finite collection $\mathcal{H}$ of subgroups of $\Gamma$ such
that $H \intersection S$ is a generating set for $H$ for each $H$ in
$\mathcal{H}$. For now we may allow the groups in $\mathcal{H}$ to be
arbitrary, although in our application they will be virtually cyclic.

In~\cite{grovesmanning08} the cusped space $X$ associated to the triple
$(\Gamma, S, \mathcal{H})$ is defined; we recall the definition here.
See~\cite{grovesmanning08} for more details and properties of the cusped space.
We use only the 1-skeleton of the cusped space so we omit the 2-cells from the
definition.

\begin{defn}\label{defn:combhoroball} For $C$ a 1-complex define the
\emph{combinatorial horoball} $\widehat{C}$ based on $C$ to be the 1-complex with
vertex set $C^{(0)} \times \{0 \union \naturals\}$ and three types of edges:
\begin{enumerate}
  \item One \emph{horizontal edge} from $(v, 0)$ to $(w, 0)$ for each edge from
    $v$ to $w$ in $C$. (Note that we allow vertices to be connected by multiple edges.)
  \item For $k \geq 1$ a \emph{horizontal edge} from $(v, k)$ to $(w, k)$
    whenever $0 < d(v, w) \leq 2^k$.
  \item A \emph{vertical edge} from $(v, k)$ to $(v, k+1)$ for each $v \in
    C^{(0)}$ and each $k \in \integers_{\geq 0}$.
\end{enumerate}
We define a height function  $h \colon \widehat{C} \to \reals_{\geq 0}$ sending a
vertex $(v, k)$ to $k$ and interpolating linearly along edges. \end{defn}

\begin{defn}\label{defn:cuspedspace} If $\Gamma$, $S$ and $\mathcal{H}$ are as
above then the \emph{cusped space} $X$ associated to this triple is defined
as follows. For $H$ in $\mathcal{H}$ let $T_H$ be a left transversal of $H$
in $\Gamma$. For $H$ in $\mathcal{H}$ and $t$ in $T_H$ let $C_{H, t}$ be the
full subgraph of $\Cay(\Gamma, S)$ containing $tH$; note that this is
isomorphic to $\Cay(H, H \intersection S)$. Then let $X$ be the union
\begin{align*}
  \Cay(\Gamma, S) \union \bigunion_{H\in\mathcal{H}, t \in T_H} \widehat{C_{H, t}}
\end{align*}
where we identify $C_{H, t} \subset \Cay(\Gamma, S)$ with the part of
$\widehat{C_{H, t}}$ with height 0. Then the height function defined on each
horoball extends to a height function $h \colon X \to \reals_{\geq 0}$. 

For $k \geq 0$ the \emph{$k$-thick part} $X_k$ of $X$ is defined to be $h^{-1}[0, k]$.
We say that a path $\gamma$ in $X$ is \emph{vertical} if $h \composed \gamma$
is strictly monotonic and \emph{horizontal} if $h \composed \gamma$ is
constant. \end{defn} 

\begin{defn} A (quasi-)isometric embedding $\alpha$ from a subinterval of
$\reals$ to $X$ is called a \emph{(quasi)-geodesic}. If that interval is
closed and bounded then $\alpha$ is a \emph{segment}. If the interval is $[0,
\infty)$ then $\alpha$ is a \emph{ray}. If the interval is all of $\reals$
then $\alpha$ is \emph{bi-infinite}.

$X$ is called a \emph{geodesic space} if for any $x$ and $y$ in $X$ there
exists a geodesic segment $\alpha \colon [a, b]$ with $\alpha(a) = x$ and
$\alpha(b) = y$.\end{defn}

\begin{defn}\label{defn:hyperbolic} A geodesic metric space $X$ is
\emph{$\delta$-hyperbolic} if every edge of any geodesic triangle in $X$ is
contained in a $\delta$-neighbourhood of the union of the other two edges.
\end{defn}

\begin{defn}\label{defn:relativelyhyperbolic} $\Gamma$ is \emph{hyperbolic
relative to $\mathcal{H}$} if the cusped space associated to $(\Gamma, S,
\mathcal{H})$ is hyperbolic for some (equivalently for any) generating set
$S$ in which $S \intersection H$ is a generating set for $H$ for each $H \in
\mathcal{H}$. \end{defn}

This definition of relative hyperbolicity is shown to be equivalent to several
other standard definitions in~\cite{grovesmanning08}.

We recall a lemma from~\cite{bowditch12} that characterises hyperbolicity of
the pair $(\Gamma, \mathcal{H})$ when $\Gamma$ is itself a hyperbolic group.
Recall that a collection $\mathcal{H}$ of subgroups of $\Gamma$ is \emph{almost
malnormal} if, for $H_1$ and $H_2$ in $\mathcal{H}$ and $g$ in $\Gamma$, $H_1
\intersection gH_2g^{-1}$ is finite unless $H_1 = H_2$ and $g \in H_1$.

\begin{lem}\label{lem:qcalmostmalnormal}\cite[Theorem 7.11]{bowditch12} Let
  $\Gamma$ be a non-elementary hyperbolic group and let $\mathcal{H}$ be a
  finite set of subgroups of $\Gamma$. Then $\Gamma$ is hyperbolic relative to
  $\mathcal{H}$ if and only if $\mathcal{H}$ is almost malnormal in $\Gamma$
and each element of $\mathcal{H}$ is quasi-convex in $\Gamma$.\end{lem}

In the cases of interest to us all groups in $\mathcal{H}$ are virtually
cyclic. A virtually cyclic subgroup of a hyperbolic group is always
quasi-convex, and is almost malnormal if and only if it is maximal among
virtually cyclic subgroups of $\Gamma$. Putting this together with
Lemma~\ref{lem:qcalmostmalnormal} we obtain the following:

\begin{lem}\label{lem:vcycperipheral} If $\Gamma$ is hyperbolic and each group
in $\mathcal{H}$ is maximal virtually cyclic then $\Gamma$ is hyperbolic
relative to $\mathcal{H}$.\end{lem}

Fix $X$ as in Definition~\ref{defn:cuspedspace} and assume that it is $\delta$-hyperbolic.
We shall require the following two lemmas describing the geometry of $X$. The
first is~\cite[Lemma 2.11]{dahmanigroves08a} and the second is the Morse lemma
for a hyperbolic metric space.

\begin{lem}\label{lem:closeray} Let $C = 3\delta$. For any $v$ in $X_0$ and $x$
in $X$ there is a geodesic ray starting at $v$ that passes within a distance
$C$ of $x$.\end{lem}

\begin{lem}\label{lem:uniformlyclose} There exists a (computable) function
$D(\lambda, \epsilon, \delta)$ such that whenever $\gamma$ is a geodesic and
$\gamma'$ is a $(\lambda, \epsilon)$-quasi-geodesic with the same end points in
$X \union \boundary X$ the Hausdorff distance from $\gamma$ to $\gamma'$ is at
most $D$. \end{lem}

Fixing $\delta$ so that $X$ is $\delta$-hyperbolic we fix the constant $C =
3\delta$ and function $D = D(\lambda, \epsilon, \delta)$ as in Lemmas
\ref{lem:closeray} and~\ref{lem:uniformlyclose} for the remainder of this and
the following section.

We shall need an algorithm to compute the constant $\delta$ with respect to
which the cusped space is $\delta$-hyperbolic. This is dealt with by the
following results.

\begin{prop}\cite[Prop.\ 2.3]{dahmani08}
\label{thm:linearisoperimetricinequality} There is an algorithm that takes as
input a presentation for a group $\Gamma$, the generators in $\Gamma$ for a
finite set of subgroups of $\Gamma$ with respect to which $\Gamma$ relatively
hyperbolic, and a solution to the word problem in $\Gamma$ and returns the
constant of a linear relative isoperimetric inequality satisfied by the given
presentation of $\Gamma$.\end{prop}

In the case of interest here, $\Gamma$ is hyperbolic. Hyperbolic groups have
uniformly solvable word problem, so the requirement that the algorithm be given
a solution to the word problem in $\Gamma$ is no restriction to its
applicability.

The linear relative isoperimetric inequality satisfied by the group is closely
related to a linear combinatorial isoperimetric inequality satisfied by the
coned-off Cayley complex. The definition of this space 
is~\cite[Definition 2.47]{grovesmanning08}.

In~\cite{grovesmanning08} the cusped 2-complex is defined; this is a simply
connected 2-complex, the 1-skeleton of which is the cusped 1-complex defined
in Definition~\ref{defn:cuspedspace}. The length of the attaching map of each
2-cell in this complex is bounded above by the maximum of $5$ and the length of
the longest relator in the given presentation for $\Gamma$.

\begin{thm}\cite[Theorem 3.24]{grovesmanning08}\label{thm:cuspedtwocomplex}
Suppose that the coned-off Cayley complex of $\Gamma$ with respect to $S$ and
$\mathcal{H}$ satisfies a linear combinatorial isoperimetric inequality with
constant $K$. Then the cusped two-complex $X$ associated to the triple
satisfies a linear combinatorial isoperimetric inequality with constant $3K(2K
+ 1)$.\end{thm}

The computation of $\delta$ from a presentation for $\Gamma$ and a set of
generators for each group $H$ in $\mathcal{H}$ is therefore completed by the
following proposition:

\begin{prop}\cite[Prop.\ 2.23]{grovesmanning08}\label{prop:linearimplieshyperbolic}
Suppose that a 2-complex $X$ is simply connected, that each attaching map has
length at most $M$ and that $X$ satisfies a linear combinatorial
isoperimetric inequality. Then the 1-skeleton $X^{(1)}$ of $X$ is
$\delta$-hyperbolic for some $\delta$ and this $\delta$ is computable from
$M$ and the constant of the isoperimetric inequality. \end{prop}

\subsection{The Bowditch boundary}

\begin{defn}\label{defn:gromovboundary} Let $X$ be a hyperbolic geodesic metric
space. The \emph{Gromov boundary $\boundary X$} of $X$ is defined to be the
quotient of the set of geodesic rays starting at some chosen base point in
which two rays are identified if they are a finite Hausdorff distance apart.
It is endowed the quotient of the compact-open topology.\end{defn}

This definition is quasi-isometry invariant and independent of the chosen base
point up to homeomorphism. 

\begin{defn}\label{defn:bowditchboundary} Let $\Gamma$ be a group that is hyperbolic
relative to a collection $\mathcal{H}$ of subgroups. The \emph{Bowditch
boundary} $\boundary(\Gamma, \mathcal{H})$ is defined to be the Gromov boundary
of the cusped space associated to $\Gamma$ and $\mathcal{H}$ with some
generating set.\end{defn}

This definition of the Bowditch boundary is shown to be equivalent to other
definitions in~\cite{bowditch12}. If each group in $\mathcal{H}$ is maximal
virtually cyclic then we have the following description of the Bowditch
boundary from~\cite{manning15}:

\begin{lem}\label{lem:bowditchfromgromov} If $\Gamma$ is hyperbolic and each
group in $\mathcal{H}$ is maximal virtually cyclic then $\boundary(\Gamma,
\mathcal{H})$ is the quotient of $\boundary(\Gamma, \emptyset)$ by the
equivalence relation in which $x \sim y$ if and only if either $x = y$ or
$\{x, y\} = g\cdot\Lambda H$ for some $g$ in $\Gamma$ and $H$ in $\mathcal{H}$.
The topology on $\boundary (\Gamma, \mathcal{H})$ is the quotient of the
topology on $\boundary(\Gamma, \emptyset)$.
\end{lem}

The boundary of a hyperbolic metric space has a natural quasi-conformal class
of metrics.  Recall the definition of the Gromov product of a pair of points
$p$ and $q$ in $X$ with respect to a base point $v$ in $X$:
\begin{align*}
\gromprod[v]{p}{q} = \frac{1}{2}(d(v, p) + d(v, q) - d(p, q)).
\end{align*}
This definition is extended to allow $p$ and $q$ to be in $X \union \boundary X$ by
\begin{align*}
\gromprod[v]{p}{q} = \sup\liminf_{i, j \to \infty}\gromprod[v]{p_i}{q_j}.
\end{align*}
Here the supremum is taken over pairs of sequences $(p_i)$ and $(q_j)$ in $X$
converging to $p$ and $q$ respectively.

By~\cite[III.H.3.21]{bridsonhaefliger99} $\boundary X$ admits a visual metric
at any base point $v$; that is, a metric $d_v$ on $\boundary X$ satisfying
\begin{align*}
k_1 a^{-\gromprod[v]{p}{q}} \leq d_v(p, q) \leq k_2 a^{-\gromprod[v]{p}{q}}.
\end{align*}
Here $a$, $k_1$ and $k_2$ can be taken to be $2^{1/4\delta}$, $3-2\sqrt{2}$ and
$1$ respectively.

\subsection{The double-dagger condition}

If the boundary of $X$ does not contain a cut point then its local
connectivity is controlled by the so-called double-dagger condition. This was
defined first in the absolute case in~\cite{bestvinamess91} and later in the
relative case in~\cite{dahmanigroves08a}. We now record the definition of the
condition and some important properties. Fix a base point $v$ in $X_0$.

Let $M = 6(C + 45\delta) + 2\delta + 3$. For $\epsilon \geq 0$ we say that a
pair of points $x$ and $y$ in $X$ satisfy $\star_\epsilon$ if $\vert d(v, x) -
d(v, y)\vert \leq \epsilon$ and $d(x, y) \leq M$. 

For $n \geq 0$ we say that a pair of points $x$ and $y$ satisfying
$\star_\epsilon$ satisfy $\ddag(\epsilon, n)(x, y)$ if there exists a path of
length at most $n$ from $x$ to $y$ that avoids the ball of radius $m - C -
45\delta + 3\epsilon$ centred at $v$, where $m = \min\{d(v, x), d(v, y)\}$.
After this section we will not need to allow $\epsilon$ to be non-zero; we
will say that $X$ satisfies $\ddag_n$ if $\ddag(0, n)(x, y)$ holds for all $x$
and $y$ in $X$ satisfying $\star_0$. However, the full definition is required
in the proof of Proposition~\ref{prop:Xsatisfiesddag}.

We require the following proposition, which is essentially Proposition 5.1
of~\cite{dahmanigroves08a}.

\begin{prop}\label{prop:Xsatisfiesddag} Suppose that $\boundary(\Gamma,
\mathcal{H})$ is connected and without a cut point. Then there exists $n$
such that $X$ satisfies $\ddag_n$. Furthermore, there is an algorithm that
computes such an $n$ if $\boundary(\Gamma, \mathcal{H})$ is connected and
without a cut point and that does not terminate if it is not connected.
\end{prop}

For clarity we recall the proof of this proposition
from~\cite{dahmanigroves08a} here.

\begin{proof} We begin by recalling some constants
from~\cite{dahmanigroves08a}. In this paper these constants will only appear
in this proof. Let $k = 2M$, let $K = 3(2^{2M+3}) + M + 3$ and for any $n$
let $R(n) = 4(n + M) + 3k + 50\delta + 3$. For each $n \geq K$ in turn, check
whether $\ddag(10\delta, n)(x, y)$ holds for all pairs of vertices $(x, y)$
in $\Ball_{R(n)}(v) \intersection X_k$ satisfying $\star_{10\delta}$. 

If $\boundary(\Gamma, \mathcal{H})$ is connected and does not contain a cut
point then such an $n$ exists by~\cite[Lemma 4.2]{dahmanigroves08a}. It is
commented that $\ddag(10\delta, n)(x, y)$ holds for all pairs of vertices $(x,
y)$ in $\Ball_{R(n)}(v)$ satisfying $\star_{10\delta}$ with $x \notin X_k$ in the
first paragraph of the proof of~\cite[Lemma 2.16]{dahmanigroves08a}.
Then~\cite[Corollary 4.6]{dahmanigroves08a} says that $X$ satisfies $\ddag_n$; note
that this corollary applies because the conclusion 
of~\cite[Lemma 2.16]{dahmanigroves08a} holds in the case of virtually cyclic
peripheral subgroups.

If $\boundary(\Gamma, \mathcal{H})$ is disconnected then there is no $n$ such
that the condition $\ddag_n$ holds by~\cite[Lemma~4.1]{dahmanigroves08a}.  \end{proof}

\subsection{Cut points and pairs}\label{sec:annulus} 

We now investigate cut points and pairs in $\boundary X$ under the assumption
that $X$ satisfies $\ddag_n$; we assume that this condition holds for the
remainder of this section. We relate the existence of such a point or pair to
the connectedness of thickened cylinders around quasi-geodesics in $X$.

Let $\gamma$ be a bi-infinite $(\lambda, \epsilon)$-quasi-geodesic in $X$ or a
$(\lambda, \epsilon)$-quasi-geodesic ray passing through $X_0$ in $X$.
Fix a point $v$ in $\gamma \intersection X_0$. Since $\gamma$ is a
quasi-geodesic, its limit set, which we shall denote $\Lambda\gamma$, is either
a pair of points $\gamma(\pm\infty)$ or a single point $\gamma(\infty)$
depending on whether $\gamma$ is bi-infinite or a ray.

Recall the definitions $C = 3\delta$ and $D = D(\lambda, \epsilon, \delta)$ as
in Lemma~\ref{lem:uniformlyclose}. 

We define a subset of $X$, the connectivity of which will be seen to
reflect the connectivity of $\boundary X - \Lambda \gamma$.  For $R \geq 0$ let
$N_R(\gamma)$ be the closed $R$-neighbourhood of $\gamma$ and for $0 \leq r
\leq R \leq \infty$ let $N_{r, R}(\gamma)$ be $\{x \in X \colon r \leq d(x,
\gamma) \leq R\}$. For $K \geq 0$ let $C_K(\gamma)$ be $\{x \in X \colon d(x,
\gamma) = K\}$. Finally, for $0 \leq r \leq K \leq R \leq \infty$ let $A_{r, R,
K}(\gamma)$ be the union of those connected components of $N_{r, R}(\gamma)$
that meet $C_K(\gamma)$. Note that $A_{r, R, K}(\gamma)$ contains $N_{K,
R}(\gamma)$.

For a component $U$ of $A_{r, \infty, K}(\gamma)$ define its \emph{shadow}
$\shadow U$ to be the set of points $p$ in $\boundary X$ such that for any
geodesic ray $\alpha$ from $v$ to $p$, $\alpha(t)$ is in $U$ for $t$ sufficiently
large. 

\begin{lem}\label{lem:disjointcoverbyshadows} Let $r > D$. Then
$\bigunion_U\shadow U = \boundary X - \Lambda\gamma$, where the union is
taken over the set of connected components of $A_{r, \infty, K}(\gamma)$.
Furthermore, $\shadow U \intersection \shadow V = \emptyset$ for distinct
components $U$ and $V$.
\end{lem}

\begin{proof} The statement that shadows are disjoint is clear from the
  definition. The assumption that $r > D$ ensures that when $U$ is a subset of
  $A_{r, \infty, K}(\gamma)$, $\shadow U$ does not contain a point in
  $\Lambda\gamma$ by Lemma~\ref{lem:uniformlyclose}. 

  If $p$ is any point in $\boundary X - \Lambda\gamma$ then any geodesic ray
  $\alpha$ from $v$ to $p$ diverges arbitrarily far from $\gamma$, so
  $d(\alpha(t), \gamma) \geq r + D$ for $t$ at least some number $t_0$. Let $U$
  be the component of $A_{r, \infty, K}(\gamma)$ that contains $\alpha(t)$ for
  $t \geq t_0$.
  If $\alpha'$ is another geodesic ray from $v$ to $p$ then for any $t$ there
  exists $t'$ such that $d(\alpha'(t), \alpha(t')) \leq D$ and then $\mod{t -
  t'}$ is guaranteed to be at most $D$. If $t \geq t_0 + D$ then $t' \geq t_0$,
  so $\alpha'(t)$ is in the same component of $A_{r, \infty, K}(\gamma)$ as
  $\alpha(t')$, which is in $U$. Therefore $\alpha'(t) \in U$ for $t \geq t_0 +
  D$. It follows that $p \in \shadow U$.\end{proof}

\begin{lem}\label{lem:nonemptyshadows} Let $U$ be a component of $A_{r, \infty,
K}(\gamma)$. Then $\shadow U$ is non-empty as long as $K \geq r + D + \delta +
C$ and $r > D$.  \end{lem}

\begin{proof} Let $x \in C_K(\gamma) \intersection U$. Then by
Lemma~\ref{lem:closeray} there exists a geodesic ray $\alpha$ from $v$ and $t
\geq 0$ such that $d(\alpha(t), x) \leq C$. Any geodesic segment in $X$ from
$x$ to $\alpha(t)$ is contained in $U$, so $\alpha(t) \in U$. Also
$d(\alpha(t), \gamma) \geq r + D + \delta$; it follows from
Lemma~\ref{lem:uniformlyclose}, hyperbolicity of $X$ and the assumption that
$r > D$ that $d(\alpha(t'), \gamma) \geq r$ for $t' \geq t$, and therefore
that $\alpha(t') \in U$ for $t' \geq t$. As in the proof of
Lemma~\ref{lem:disjointcoverbyshadows} it follows from this and the fact that
$r > D$ that $\alpha(\infty) \in \shadow U$.\end{proof}

\begin{lem}\label{lem:shadowsclopen} If $U$ is a component of $A_{r, \infty,
K}(\gamma)$, $\shadow U$ is closed and open in $\boundary X - \Lambda\gamma$ as
long as $r > D$. \end{lem}

\begin{proof} Let $p$ be a point in $\shadow U$ and let $p = \alpha(\infty)$
where $\alpha$ is a geodesic ray from $v$. For $t \geq 0$ let
$V_{t}(\alpha)$ be the set of end points of geodesic rays $\beta$ from $v$ such that
$d(\beta(t), \alpha(t)) < 2\delta + 1$. The collection of such sets as $t$
varies forms a fundamental system of neighbourhoods of $p \in \boundary X$. 

Then there exists $t_0$ such that for $t \geq t_0$, $d(\alpha(t), \gamma) \geq
r + 2D + 7\delta + 1$.  We claim that $V_{t_0}(\alpha) \subset \shadow U$
for $t_0$ as defined in the previous paragraph. To see this, let $q \in
V_{t_0}(\alpha)$ and let $\beta$ be a geodesic ray from $v$ to $q$, so
$d(\beta(t_0), \alpha(t_0)) < 2\delta+1$.  Let $\beta'$ be another geodesic ray
from $v$ with $\beta'(\infty) = q$, so $d(\beta'(t_0), \beta(t_0)) < 4\delta$.
Suppose that there exists $t \geq t_0$ such that $\beta'(t) \notin U$. Then
there exists $t' \geq t_0$ such that $d(\beta'(t'), \gamma) \leq r$; without
loss of generality assume that $d(\beta'(t'), \gamma\restricted{[0, \infty)})
\leq r$. Let $\gamma'$ be a geodesic ray from $v$ to $\gamma(\infty)$, so the
Hausdorff between $\gamma\restricted{[0, \infty)}$ and $\gamma'$ is at most
$D$. Then $d(\beta'(t'), \gamma') \leq r + D$ and $d(\beta'(t_0), \gamma')
\leq r + D + \delta$. Putting these inequalities together, $d(\alpha(t_0),
\gamma) \leq r + 2D + 7\delta + 1$, which is a contradiction. Hence
$\beta(t) \in U$ for $t \geq t_0$, and so $q \in \shadow U$.

Since $p$ was arbitrary in $\shadow U$, it follows that $\shadow U$ is open. As
$U$ ranges over the connected components of $A_{r, \infty, K}(\gamma)$,
$\shadow U$ ranges over a cover of $\boundary X - \Lambda\gamma$ by disjoint
open subsets (since $r > D$), so each is also closed.\end{proof}

The proof of the following lemma is based on the proof 
of~\cite[Proposition 3.2]{bestvinamess91}.

\begin{lem}\label{lem:shadowsconnected} If $U$ is a connected component of 
$A_{r, \infty, K}(\gamma)$ then $\shadow U$ is contained in one connected
component of $\boundary X - \Lambda \gamma$ as long as $r$ satisfies the
following inequality.
\begin{align*}
  r > 2\log_a\left(\frac{k_2}{k_1}\frac{n-1}{1-a^{-1}}\right) + M + 12\delta + D.
\end{align*}
\end{lem}

\begin{proof} Let $p$ and $q$ be points in $\shadow U$ and let $\alpha_1$ and
$\alpha_2$ be geodesic rays from $v$ to $p$ and $q$ respectively. Then there
exist $t_1$ and $t_2$ such that $\alpha_1(t_1)$ and $\alpha_2(t_2)$ are in
$U$. Let $\phi\colon [0,\ell] \to X$ be a path in $U$ parametrised by arc
length connecting the two points $\alpha_1(t_1)$ and $\alpha_2(t_2)$.

For each integer $i$ in $[0,\ell]$ let $z_i$ be a point within a distance $C$
of $\phi(i)$ so that there is a geodesic ray $\beta_i$ from $v$ with
$\beta_i(m_i) = z_i$; we can assume that $\beta_0 = \alpha_1$ and $\beta_\ell =
\alpha_2$.  Then, following the argument of~\cite[Prop.\ 3.2]{bestvinamess91},
we show that $r_i(\infty)$ and $r_{i+1}(\infty)$ can be connected in $\boundary
X - \Lambda\gamma$ for each $i$; for notational convenience we prove it for
$i=0$.

Using the condition $\ddag_n$, define geodesic rays $\beta_t$ for each
$n$-adic rational $t$ in $[0, 1]$ inductively on the power $k$ of the
denominator of $t$ to satisfy:
\begin{align*} 
  d\left(\beta_{j/n^k}(m_i + k), \beta_{j+1/n^k}(m_i + k)\right) \leq M 
  \text{ for each $j$ with $0 \leq j \leq n^k-1$}.
\end{align*}
Note that the first step of the induction holds since $M$ is at least $2C+1$.
The triangle inequality gives the following lower bound on the Gromov product
of these points.
\begin{align*}
  \gromprod[v]{\beta_{j/n^k}(\infty)}{\beta_{j+1/n^k}(\infty)} & \geq
\liminf_{n_1,n_2}\gromprod[v]{\beta_{j/n^k}(n_1)}{\beta_{j+1/n^k}(n_2)}\\
& \geq \gromprod[v]{\beta_{j/n^k}(m_0 + k)}{\beta_{j+1/n^k}(m_0 + k)}\\
& = m_0 + k - M/2
\end{align*}
Let $d_v$ be a visual metric on $\boundary X$ with base point $v$, visual
parameter $a$ and multiplicative constants $k_1$ and $k_2$. We obtain:
\begin{align*}
  d_v\left(\beta_{j/n^k}(\infty), \beta_{j+1/n^k}(\infty)\right) \leq k_2a^{-m_0 - k + M/2}.\tag{*}
\end{align*}

Inductively applying the triangle inequality we arrive at the following inequality
\begin{align*}
  d_v\left(\beta_0(\infty), \beta_t(\infty)\right) \leq \frac{k_2(n-1)a^{-m_0 +
  M/2}}{1-a^{-1}}\quad\text{for each $n$-adic rational $t \in [0, 1]$.}
\end{align*}
Define a path $\psi \colon [0,1] \to \boundary X$ with $\psi(t) =
\beta_t(\infty)$ for each $n$-adic rational $t$ in $[0, 1]$; this extends
continuously to a path from $\beta_0(\infty)$ to $\beta_1(\infty)$ by the
uniform continuity of the map $t \to \beta_t(\gamma)$ defined on the $n$-adic
rationals, which is established by equation~$(*)$. This path is
contained in the ball of radius $k_2(n-1)a^{-m_0 + M/2}/(1-a^{-1})$ around
$\beta_0(\infty)$.

We now bound below the distance $d_v(\beta_0(\infty), \Lambda\gamma)$. Let
$\gamma'$ be a geodesic ray from $v$ to $\gamma(\infty)$, so the Hausdorff
distance between $\gamma$ and $\gamma'$ is at most $D$.
By~\cite[III.H.3.17]{bridsonhaefliger99},
\begin{align*}
  \gromprod[v]{\beta_0(\infty)}{\gamma'(\infty)} \leq
  \liminf_{n_1,n_2}\gromprod[v]{\beta_0(n_1)}{\gamma'(n_2)} + 2\delta.
\end{align*}
Let $n_1$ and $n_2$ each be at least $m_0$. Certainly $d(\beta_0(m_0), \gamma') >
\delta$ since $r > \delta + D$, so there exists a point $p$ on $[\beta_0(n_1),
\gamma'(n_2)]$ within a distance $\delta$ of $\beta_0(m_0)$. In fact, $d(\beta_0(m_0),
\gamma) > 2\delta + D$, so $d(\beta_0, \gamma'(m_0)) > 2\delta$. Therefore there
exists a point $q$ on $[\beta_0(n_1), \gamma'(n_2)]$ within a distance $\delta$ of
$\gamma'(m_0)$. 

Suppose that $q$ is closer to $\beta_0(n_1)$ than $p$. Then by considering the
geodesic triangle with vertices $\beta_0(n_1)$, $\beta_0(m_0)$ and $p$ we see
that $q$ is within distance $2\delta$ of $\beta_0$, and therefore the distance
from $\gamma'(m_0)$ to $\beta_0(m_0)$ is at most $6\delta$. But we assumed that
$r > 6\delta + D$, which gives a contradiction. This implies that
$d(\beta_0(n_1), \gamma'(n_2))$ is equal to the sum of the distances
$d(\beta_0(n_1), p)$, $d(p, q)$ and $d(q, \gamma'(n_2))$.

Then we have the following inequality.
\begin{align*} \gromprod[v]{\beta_0(n_1)}{\gamma'(n_2)} - \gromprod[v]{\beta_0(m_0)
    }{\gamma'(m_0)} 
    & = d(\beta_0(n_1), \beta_0(m_0)) - d(\beta_0(n_1), p) \\
    & + d(\beta_0(m_0), \gamma'(m_0)) - d(p, q) \\
    & + d(\gamma'(m_0), \gamma'(n_2)) - d(q, \gamma'(n_2))\\
    & \leq \delta + 2\delta + \delta = 4\delta
\end{align*}

This implies a lower bound on the distance from $\beta_0(\infty)$ to
$\gamma(\infty)$ with respect to the visual metric.
\begin{align*}
  d_v(\beta_0(\infty), \gamma(\infty)) \geq k_1a^{-m_0 + (r-D)/2 - 6\delta}.
\end{align*}
In the case that $\gamma$ is bi-infinite, $d_v(\beta_0(\infty),
\gamma(-\infty))$ similarly satisfies the same bound.  Therefore, by the
assumption on $r$ the path constructed from $\beta_0(\infty)$ to
$\beta_1(\infty)$ avoids $\Lambda\gamma$.
\end{proof}

\begin{lem}\label{lem:replaceinftybyR} The inclusion map $A_{r, R, K}(\gamma)
\hookrightarrow A_{r, \infty, K}(\gamma)$ induces a bijection between the sets
of connected components of those subspaces of $X$ as long as $R \geq 4\delta + D +
\max\{r + 4\delta + 1, K\}$.\end{lem}

\begin{proof} Surjectivity is clearly guaranteed by the fact that $R \geq K$.
For injectivity, let $x$ and $y$ be points in $C_K(\gamma)$ that lie in the
same connected component of $A_{r, \infty, K}(\gamma)$. We show that the
shortest path from $x$ to $y$ in $N_{r, \infty}(\gamma)$ stays within a
distance $R$ of $\gamma$. Let $\phi\colon [0,\ell] \to X$ be such a
shortest path parametrised by arc length. Suppose that $d(\phi(s),\gamma) >
R$. Let $[t_0,t_1]$ be a maximal subinterval of $[0, \ell]$ containing $s$
such that $d(\phi(t),\gamma) \geq r + 4\delta + 1$ for $t \in [t_0,t_1]$. 

Then for $t \in [t_0,t_1]$, $\phi\restricted{[t-4\delta-1,t+4\delta+1]
\intersection[t_0, t_1]}$ has image in $N_{r+4\delta+1, \infty}(\gamma)$.
Therefore any geodesic segment from $\phi(\min\{t-4\delta-1, t_0\})$ to
$\phi(\max\{t+4\delta+1, t_1\})$ is contained in $N_{r, \infty}(\gamma)$, so by
minimality of the length of $\phi$, $\phi\restricted{[t-4\delta-1,t+4\delta+1]
\intersection [t_0, t_1]}$ is a geodesic. This means that
$\phi\restricted{[t_0,t_1]}$ is an $(8\delta+2)$-local geodesic. Therefore by
\cite[III.H.1.13]{bridsonhaefliger99} it is contained in a
$2\delta$-neighbourhood of any geodesic from $\phi(t_0)$ to $\phi(t_1)$.

By maximality of $[t_0,t_1]$, either $d(\phi(t_0),\gamma) = r + 4\delta+1$ or
$t_0=0$, so certainly $d(\phi(t_0),\gamma) \leq \max\{r+4\delta+1,K\}$, and
similarly $d(\phi(t_1),\gamma)$ satisfies the same inequality. By
$\delta$-hyperbolicity applied to the geodesic quadrilateral with vertices
$\phi(t_0)$, $\phi(t_1)$ and the points $\gamma(s_0)$ and $\gamma(s_1)$ on
$\gamma$ minimising the distances to $\phi(t_0)$ and $\phi(t_1)$, any geodesic
from $\phi(t_0)$ to $\phi(t_1)$ is contained in a $2\delta + \max\{r + 4\delta
+ 1, K\}$ neighbourhood of a geodesic from $\gamma(s_0)$ to $\gamma(s_1)$, so
is a subset of $N_{2\delta+\max\{r+4\delta+1, K\} + D}(\gamma)$. Hence
$d(\phi(s),\gamma) \leq 4\delta + \max\{r+8\delta+2,K\} + D$, which is a
contradiction.  \end{proof}

From the results of this section we conclude the following:

\begin{prop}\label{prop:summaryofannulusresults} The map that sends a component
$U$ of $A_{r, R, K}(\gamma)$ to the shadow (with respect to some base point
$v \in \gamma$) of the component of $A_{r, \infty, K}(\gamma)$ containing $U$
is a well defined bijection between the set of connected components of $A_{r,
R, K}(\gamma)$ and the set of connected components of $\boundary X -
\Lambda\gamma$ as long as $r$, $R$ and $K$ are taken to simultaneously
satisfy the conditions of lemmas~\ref{lem:disjointcoverbyshadows},
\ref{lem:nonemptyshadows}, \ref{lem:shadowsclopen}, \ref{lem:shadowsconnected}
and~\ref{lem:replaceinftybyR}.\qed\end{prop}

\begin{rem}\label{rem:rKRcomputable} The conditions on $r$, $R$ and $K$ depend
only on $\delta$, $n$, $\lambda$ and $\epsilon$ and suitable values can be
computed from these data.\end{rem}

We end this section with the following lemma, which shows that connectedness of
$A_{r, R, K}(\gamma)$ can be detected locally.

\begin{lem}\label{lem:Adisconnectedlocally} Suppose that $\gamma$ is a
bi-infinite $(\lambda, \epsilon)$-quasi-geodesic and that neither point in
$\Lambda\gamma$ is a cut point. Let $r$ and $R$ be chosen to satisfy
lemmas~\ref{lem:shadowsconnected} and~\ref{lem:replaceinftybyR}. Let $T$ be
at least $\log_a(2k_1/k_2) + 3D + 2\delta + K$. Then every component of $A_{r, R,
K}(\gamma)$ meets $C_K(\gamma) \intersection \Ball_T(\gamma(t))$ for any $t$
such that $\gamma(t)$ is in $X_0$.  \end{lem}

\begin{proof} Fix the base point $v = \gamma(t)$.
Let $U$ be a component of $A_{r, R, K}(\gamma)$ and let $U'$ be the component
of $A_{r, \infty, K}(\gamma)$ containing $U$.
Then it is sufficient to show that $U'$ meets $C_K(\gamma) \intersection
\Ball_T(v)$ since $U' \intersection C_K(\gamma) = U \intersection C_K(\gamma)$
by Lemma~\ref{lem:replaceinftybyR}. 

Suppose that $U'$ does does not meet $C_K(\gamma) \intersection \Ball_T(v)$.
Let $p$ be a point in $\shadow U$ and let $\alpha$ be a geodesic ray from $v$
to $p$.
Then $\alpha(s) \in U \intersection C_K(\gamma)$ for some $s$; by assumption
$d(\alpha(s), v) \geq T$.

Let $\gamma'$ be a geodesic connecting the points of $\Lambda\gamma$ so that
the Hausdorff distance between $\gamma$ and $\gamma'$ is at most $D$.
Parametrise $\gamma'$ so that $d(\gamma'(0), v) \leq D$.
Then $d(\alpha(s), \gamma') \leq K + D$; let $d(\alpha(s), \gamma'(s')) \leq K
+ D$.
This implies that $d(\gamma'(s'), v)\geq T - D - K$.
We therefore have
\begin{align*}
  \gromprod[v]{\alpha(s)}{\gamma'(s')} \geq T - D - K.
\end{align*}
Assume that $s'\geq0$; this implies that
\begin{align*}
  \gromprod[v]{p}{\gamma'(\infty)} & \geq \liminf_{m, n \to \infty}\gromprod[v]{\alpha(m)}{\gamma'(n)} \\
                                   & \geq \gromprod[v]{\alpha(s)}{\gamma'(s')} - D \\
                                   & \geq T - 2D - K.
\end{align*}
So $d_v(p, \gamma'(\infty)) \leq k_2a^{-T + 2D + K}$. Similarly, if $s' \leq 0$, $d_v(p,
\gamma'(-\infty)) \leq k_2a^{-T + 2D + K}$. Therefore $\shadow U'$
is contained in a $k_2a^{-T + 2D + K}$ neighbourhood of $\Lambda\gamma$.  Also,
for any $s$, the geodesic from $\gamma(t+s)$ to $\gamma(t-s)$ passes within a
distance $D$ of $\gamma(t)$, so $\gromprod[v]{\gamma(t+s)}{\gamma(t-s)} \leq
D$. Then $\gromprod[v]{\gamma(\infty)}{\gamma(-\infty)} \leq 2\delta + D$, and
so $d_v(\gamma(\infty), \gamma(-\infty)) \geq k_1a^{-(2\delta + D)}$. It follows by
the inequality satisfied by $T$ that the closed balls of radius $k_2a^{-T + 2D
+ K}$ around $\gamma(\infty)$ and $\gamma(-\infty)$ are disjoint. By
Lemma~\ref{lem:shadowsconnected} $\shadow U'$ is connected, so is 
contained in one of these two balls, say in the ball around $\gamma(\infty)$.
But then $\shadow U$ is a non-empty proper subset of $\boundary X -
\{\gamma(\infty)\}$ that is closed and open, so $\gamma(\infty)$ is a cut
point, which is a contradiction.\end{proof}

\section{Detecting cut points and pairs}\label{sec:cutpairsalgorithms}

We now use the results of the previous section to prove some computability
results concerting topological features of the Bowditch boundary of a hyperbolic group
under the assumption that the cusped space satisfies a double dagger condition. These
are the main technical results of this paper. The idea is to identify the
topological feature of the boundary with a combinatorial feature of the cusped
space of bounded size, so that the existence of that feature can be
determined by looking at only a finite part of the the cusped space.

Let $\Gamma$ be a group hyperbolic relative to a finite set $\mathcal{H}$ of
maximal virtually cyclic subgroups. Let $S$ be a generating set for $\Gamma$ such that
$S \intersection H$ generates $H$ for each $H$ in $\mathcal{H}$. Let $X$ be the
cusped space associated to $(\Gamma, \mathcal{H}, S)$.

\subsection{Geodesics in horoballs}\label{sec:horoballs}

First we must understand the connectivity of neighbourhoods of geodesics in the
thin part of $X$.  We assume that the peripheral subgroups are virtually cyclic, so
the geometry of the cusps of $X$ is relatively simple. For notational
convenience we initially restrict to the case in which $X$ consists of a single
cusp. Then the vertex set of $X$ can be identified with $H \times
\integers_{\geq 0}$ and for any $k$ there is an inclusion of the Cayley graph $\Cay(H, S)
\hookrightarrow h^{-1}(k)$ mapping a vertex $h \in \Cay(H, S)$ to $(h, k)$; for
$k = 0$ this inclusion is an isomorphism of graphs. 

Let $d_H$ be the word metric in $\Cay(H, S)$ and let $\Cay(H, S)$ be
$\delta_H$-hyperbolic with respect to this word metric. Let $\alpha$ be a
bi-infinite geodesic in $\Cay(H, S)$ with respect to $d_H$; then any point in
$\Cay(H, S)$ is within a distance of at most $2\delta_H + 1$ of $\alpha$.

Let $\gamma \colon [0, \infty) \to X$ be a vertical geodesic ray with
$\gamma(0) = (\alpha(0), 0)$. Then for any $k \in \integers_{\geq 0}$, the
vertex set of $h^{-1}(k) \intersection N_{r, R}(\gamma)$ is
\begin{align*}
  \{g \in H \colon 2^{k + r - 1} \leq d_H(g, \alpha(0)) \leq 2^{k + R - 1}\} \times \{k\}.
\end{align*} 
We will denote by $Y_k$ the set $h^{-1}(k) \intersection N_{r, R}(\gamma)$.

Assume now that $k \geq \log_2(2\delta_H + 1)$. Then every vertex in
$h^{-1}(k)$, and therefore every vertex in $Y_k$, is adjacent to a vertex in
$\alpha \times \{k\}$. Therefore $Y_k$ contains connected components $Y_k^+$
and $Y_k^-$ with
\begin{align*}
  \alpha\restricted{[2^{k+r-1}, 2^{k+R-1}]} \times \{k\} & \subset Y_k^+,\\
  \alpha\restricted{[-2^{k+R-1}, -2^{k+r-1}]} \times \{k\} & \subset Y_k^-,\\
\end{align*}
and each of these components meets $C_K(\gamma)$. Therefore each of the sets
$Y_k^\pm$ is a subset of a component of $A_{r, R, K}(\gamma)$. Any vertex in
the complement of these two components of $Y_k$ is contained in 
\begin{align*}
  \{g \in H \colon d_H(g, \alpha(0)) \leq 2^{k + r}\} \times \{k\}.
\end{align*}
Therefore only those vertices of $Y_k$ that are in $Y_k^+ \union Y_k^-$ are
adjacent in $X$ to vertices of $Y_{k+1}$. Furthermore, $Y_k^+$ is adjacent to
$Y_{k+1}^+$ and not to $Y_{k+1}^-$ and likewise for $Y_k^-$. Finally, vertices
that are in $Y_{k+1}$ but not in $Y_{k+1}^\pm$ are adjacent to vertices in
$Y_k$.

Thus, if $k \geq \log_2(2\delta_H + 1)$ then $A_{r, R, K}(\gamma) \intersection
h^{-1}[k, \infty)$ contains two unbounded components $Y_{\geq k}^+$ and
$Y_{\geq k}^-$ containing $\union_{l \geq k} Y_l^+$ and $\union_{l \geq k}
Y_l^-$ respectively and the complement of these two components is contained in
\begin{align*}
  \{g \in H \colon d_H(g, \alpha(0)) \leq 2^{k + r}\} \times \{k\}.
\end{align*}

To make precise the consequences of this description of $A_{r, R, K}(\gamma)$,
we make the following definition, now allowing $X$ to consist of more than a
single cusp. Let $\gamma\colon [a, b] \to X$ be a geodesic segment such that
$h(\gamma(a)) = h(\gamma(b)) = k > R$ such that $h \composed \gamma$ is decreasing
at $a$ and increasing at $b$. Let $\hat\gamma \colon (-\infty, \infty) \to X$
be the path obtained by concatenating $\gamma$ with two vertical geodesic
rays.  Note that this is a $k$-local-geodesic. Let $A_{r, R, K}'(\gamma)$ be
\begin{align*}
A_{r, R, K}(\hat\gamma) - h^{-1}[k, \infty) \intersection
  N_R\left(\hat\gamma (-\infty, a] \union \hat\gamma[b, \infty)\right).
\end{align*} 

The results of this section together give the following lemma, which we shall
use to control the depth to which geodesics in $X$ connecting cut pairs in
$\boundary X$ penetrate into the thin part of $X$.

\begin{lem}\label{lem:geodesicsinhoroballs} Let $\gamma$ and $\hat\gamma$ be as
in the previous paragraph. Let $\delta_\mathcal{H}$ be such that each $H \in
\mathcal{H}$ is $\delta_\mathcal{H}$-hyperbolic with respect to the
generating set $H \intersection S$.  Suppose that $h(\gamma(a)) =
h(\gamma(b)) \geq \min\{R, \log_2(2\delta_\mathcal{H} + 1)\}$. Then the
inclusion $A'_{r, R, K}(\gamma) \hookrightarrow A_{r, R, K}(\hat\gamma)$
induces a bijection between the sets of connected components of those
spaces.\qed\end{lem}

\subsection{Cut points}

We now show that there is an algorithm that determines whether or not
$\boundary X$ contains a cut point under the assumption that $X$
satisfies a double dagger condition.

\begin{prop}\label{prop:cutpointscomputable} There is an algorithm that takes
as input a presentation for a hyperbolic group $\Gamma$ with generating set
$S$, a list of subsets of $S$ generating a collection $\mathcal{H}$ of maximal
virtually cyclic subgroups of $\Gamma$ and integers $\delta$ and $n$ such that the
cusped space is $\delta$-hyperbolic and satisfies $\ddag_n$ and returns the
answer to the question ``does $\boundary(\Gamma, \mathcal{H})$ contain a cut
point?'' \end{prop}

\begin{proof} It is shown in~\cite[Theorem 0.2]{bowditch99a} that any cut point in
$\boundary(\Gamma, \mathcal{H})$ must be the limit point of $gHg^{-1}$ for some
$H \in \mathcal{H}$ and $g \in \Gamma$. Therefore it is sufficient to check
whether or not $\Lambda H$ is a cut point for each $H \in \mathcal{H}$.  Choose
$r$, $R$ and $K$ to simultaneously satisfy the conditions of
lemmas~\ref{lem:disjointcoverbyshadows}, \ref{lem:nonemptyshadows},
\ref{lem:shadowsclopen}, \ref{lem:shadowsconnected}
and~\ref{lem:replaceinftybyR} with $\lambda = 1$ and $\epsilon = 0$. Let
$\delta_\mathcal{H}$ be large enough that each $H$ in $\mathcal{H}$ is
$\delta_\mathcal{H}$-hyperbolic with respect to the generating set $S \intersection H$ and
let $k \geq \log_2(2\delta_\mathcal{H} + 1)$.

Then for each vertical geodesic ray $\gamma$ starting at the identity element
$1 \in X_0$ check whether or not $A_{r, R, K}(\gamma) \intersection h^{-1}([0,
k])$ is connected; as in Lemma~\ref{lem:geodesicsinhoroballs}, it is connected
if and only if $A_{r, R, K}(\gamma)$ is connected. Also, $A_{r, R, K}(\gamma)$
is disconnected if and only if $\Lambda\gamma = \Lambda H$ is a cut point by
lemmas~\ref{lem:disjointcoverbyshadows}, \ref{lem:nonemptyshadows},
\ref{lem:shadowsclopen}, \ref{lem:shadowsconnected}
and~\ref{lem:replaceinftybyR}, where $H$ is the element of $\mathcal{H}$ such
that the combinatorial horoball based on $H \subset \Cay(X, S)$ contains
$\gamma$. This check can be completed in finite time. \end{proof}

\subsection{Cut pairs}\label{subsec:cutpairs}

We now assume that $X$ is $\delta$-hyperbolic and satisfies $\ddag_n$, that
$\boundary X$ contains no cut point and that the Cayley graph of $H$ with
respect to its generating set $S \intersection H$ is $\delta_\mathcal{H}$-hyperbolic for
each $H$ in $\mathcal{H}$.  Then all results of sections~\ref{sec:annulus}
and~\ref{sec:horoballs} can be applied.

First we show that the existence of a cut pair in $\boundary X$ is equivalent
to the existence of a feature in $X$ of known bounded size. Then by searching
for such a feature one can determine whether or not $\boundary X$ contains a
cut pair. We use a pumping lemma argument: we aim to replace an
arbitrary geodesic joining the two points in a cut pair in $\boundary X$ with a
periodic quasi-geodesic with bounded period that also joins the points of a
(possibly different) cut pair.

Before stating the proposition we define some constants. Take $\lambda$ and
$\epsilon$ so that any $(8\delta + 1)$-local-geodesic is a $(\lambda,
\epsilon)$-quasi-geodesic; for example let $\lambda = (12\delta + 1)/(5\delta +
1)$ and let $\epsilon = 2\delta$. Fix $r$, $R$ and $K$ to simultaneously
satisfy the conditions of propositions~\ref{lem:disjointcoverbyshadows},
\ref{lem:nonemptyshadows}, \ref{lem:shadowsclopen}, \ref{lem:shadowsconnected}
and~\ref{lem:replaceinftybyR} and fix $T$ to satisfy the conditions of
Lemma~\ref{lem:Adisconnectedlocally} with this choice of $\lambda$ and
$\epsilon$. Also let 
\begin{align*}
     k &= \max\{8\delta + 1, \log_2(2\delta_\mathcal{H} + 1), T + R\},\\
  \rho &= (2R + \epsilon)\lambda^2 + \epsilon + R \text{ and}\\
  \eta &= \max\{(8\delta + 1)/2, \lambda(T + K) + \lambda\epsilon, \lambda(R + r) + \lambda\epsilon, \lambda(R + \rho) + \lambda\epsilon\}.
\end{align*}
Let $B$ be the maximum valence of any
vertex in $X_{k + R}$ and let $V$ be the maximum number of vertices in any ball
of radius $\rho$ around any vertex in $X_{k + R}$. Then define 
\begin{align*}
  N_\text{min} &= \max\{8\delta+1, \lambda(2R+1) + \lambda\epsilon + 1\} \text{ and}\\
  N_\text{max} &= N_\text{min} \left(k + R + 1\right) B^{2\eta} 2^V + 1.
\end{align*}

\begin{prop}\label{prop:cutpairsfeature} $\boundary X$ contains a cut pair
if and only if $X$ contains one of the following two features:
\begin{enumerate}
  \item A short period geodesic at shallow depth in $X$: a geodesic segment
    $\gamma \colon [a - \eta, b + \eta] \to X$ contained in $X_{k + R}$ such
    that
  \begin{enumerate}
    \item\label{cutpairshort} $N_\text{min} \leq b - a \leq N_{max}$,
    \item\label{cutpairdepth} $h(\gamma(a)) = h(\gamma(b))$, so there exists $g
      \in \Gamma$ such that $\gamma(b) = g \cdot \gamma(a)$,
    \item\label{cutpairlocal} $\gamma\restricted{[b - \eta, b + \eta]} = g\cdot
      \gamma\restricted{[a - \eta, a + \eta]}$,
    \item\label{cutpairnontrivial} there is a partition $\mathcal{P}$ of the
      vertices of $N_{r, R}(\gamma) \intersection
      N_R(\gamma\restricted{[a,b]})$ into two subsets such that adjacent
      vertices lie in the same subset and each of the sets meets $C_K(\gamma)
      \intersection \Ball_T(\gamma(c))$ for some $c \in [a, b]$, and
    \item\label{cutpairpartition} the partition on the vertices of $N_{r,
      R}(\gamma) \intersection \Ball_{\rho}(\gamma(b))$ induced by the restriction of
      $\mathcal{P}$ to that subset is the same as the translate by $g$ of the 
      partition on the vertices of $N_{r, R}(\gamma) \intersection
      \Ball_{\rho}(\gamma(a))$ obtained by the restriction of $\mathcal{P}$; note
      that $N_{r, R}(\gamma) \intersection \Ball_\rho(\gamma(b))$ is equal to $g
      \cdot N_{r, R}(\gamma) \intersection \Ball_\rho(\gamma(a))$ by
      condition~\ref{cutpairlocal}.
  \end{enumerate}
  \item A short horseshoe-shaped geodesic: a geodesic segment $\gamma \colon
    [a, b] \to X$ such that
  \begin{enumerate}
    \item $b - a \leq N_{max} - 2R + 2\eta$
    \item $h(\gamma(a)) = h(\gamma(b)) \geq k$,
    \item $h \composed \gamma$ is decreasing at $a$ and increasing at $b$,
    \item $A'_{r, R, K}(\gamma)$ is disconnected.
  \end{enumerate}
\end{enumerate}
\end{prop}

\begin{proof} First suppose that the first type of feature exists in $X$.
Define a path $\gamma'$ in $X$ by
\begin{align*} \gamma'\left((b - a)m + t\right) = g^m\gamma(a + t)
\end{align*}
for $m \in \integers$ and $t \in [0, b-a]$. $\gamma'$ is an $(8\delta +
1)$-local-geodesic by condition~\ref{cutpairlocal} since $\eta \geq (8\delta +
1)/2$.  It is therefore a $(\lambda, \epsilon)$-quasi-geodesic
by~\cite{bridsonhaefliger99}. We now aim to show that $A_{r, R, K}(\gamma')$ is
disconnected.

Note that $N_R(\gamma')$ is a union $\bigcup_{m\in\integers} g^m\cdot
N_R(\gamma\restricted{[a, b]})$ of translates of neighbourhoods of $\gamma$.
Since $\eta \geq \lambda(R + r) + \lambda\epsilon$, $N_r(\gamma') \intersection
N_R(\gamma\restricted{[a, b]})$ is a subset of $N_{r}(\gamma)$, so is equal to
$N_r(\gamma) \intersection N_R(\gamma\restricted{[a, b]})$.  Therefore $N_{r,
R}(\gamma')$ decomposes as a union
\begin{align*} N_{r, R}(\gamma') = \bigcup_{m \in \integers} g^m\cdot
  \left(N_{r, R}(\gamma) \intersection N_R\left(\gamma\restricted{ 
    [a,b]}\right)\right).
\end{align*}

Since $b - a > \lambda(2R+1) + \epsilon\lambda$, $g^m \cdot
N_R(\gamma\restricted{[a, b]})$ and $g^l \cdot N_R(\gamma\restricted{[a, b]})$
contain no adjacent vertices for $\mod{m - l} \geq 2$. Furthermore, if $l =
m+1$ then any pair of adjacent vertices in these two sets is contained in
$g^m\cdot \Ball_\rho(\gamma(b))$ since $\rho \geq (2R + \epsilon)\lambda^2 +
\epsilon + R$.

For each set $U \in \mathcal{P}$ define a set $U'$ of vertices of $N_{r,
R}(\gamma')$ by letting $u \in U'$ if $g^m u \in U$ for some $m \in \integers$.
This gives a well defined partition $\mathcal{P}'$ of the vertices of $N_{r,
R}(\gamma')$ such that adjacent vertices lie in the same set by condition
\ref{cutpairpartition}. Its restriction to $A_{r, R, K}(\gamma')$ is
non-trivial: $C_K(\gamma')$ contains $C_K(\gamma) \intersection \Ball_T(\gamma(c))$
since $\eta \geq \lambda(T + K) + \lambda\epsilon$ and this set meets both sets
in $\mathcal{P}'$ by condition \ref{cutpairnontrivial}; therefore $A_{r, R,
K}(\gamma')$ is disconnected and therefore $\Lambda\gamma'$ is a cut pair by
the results of section~\ref{sec:annulus}.

Now suppose that the second type of feature exists in $X$. Let $\hat\gamma$ be
the $(8\delta + 1)$-local-geodesic obtained by concatenating $\gamma$ with
vertical geodesic rays. Then $A_{r, R, K}(\hat\gamma)$ is disconnected by
Lemma~\ref{lem:geodesicsinhoroballs} and $\Lambda\hat\gamma$ is a cut pair by
the results of section~\ref{sec:annulus}.

Conversely, suppose that $\boundary X$ does contain a cut pair. Let $\gamma'$
be a geodesic in $X$ such that $\Lambda\gamma'$ is a cut pair. Assume first
that some connected component of $\gamma'^{-1} h^{-1}[0, k + R]$ is an interval of length
less than $N_{max} + 2\eta$, say $[a - R, b + R]$ with $h(\gamma'(a - R)) =
h(\gamma'(b + R)) = k + R$, so $h(\gamma'(a)) = h(\gamma'(b)) = k$. Let $\gamma
= \gamma' \restricted{[a, b]}$.  Let $c \in [a, b]$ such that $h(\gamma'(c)) =
0$.  $A_{r, R, K}(\gamma')$ is disconnected and each component meets
$C_K(\gamma') \intersection \Ball_T(\gamma'(c))$ by
Lemma~\ref{lem:Adisconnectedlocally}. Since $k \geq T$ this is a subset of
$A'_{r, R, K}(\gamma)$, and $A'_{r, R, K}(\gamma)$ is  a subset of $A_{r, R,
K}(\gamma')$, so $A_{r, R, K}'(\gamma)$ is disconnected. Therefore $\gamma'$ is
a feature of the second kind described in the proposition.

On the other hand, suppose that some interval $[-\eta, N_{max} + \eta]$ is a
subset of $\gamma'^{-1} h^{-1}[0, k + R]$. Then there exist $a_0 < a_1 < \dots
< a_{2^V}$ in $[0, N_{max}]$ such that $h(a_i) = h(a_j)$ for all $i$ and $j$,
so $a_i = g_i a_0$ for some $g_i \in \Gamma$, such that
$\gamma'\restricted{[a_i - \eta, a_i + \eta]} = g_i \cdot
\gamma'\restricted{[a_0 - \eta, a_0 + \eta]}$, and such that $a_i - a_{i - 1}
\geq N_\text{min}$. Let $\mathcal{P}'$ be a partition of the vertices of $N_{r,
R}(\gamma')$ into two subsets such that adjacent vertices are in the same set
and so that both sets meet $C_K(\gamma')$. Such a partition
exists by the results of section~\ref{sec:annulus} since $\Lambda\gamma'$ is a
cut pair. Since $\eta \geq \lambda(R + \rho) + \lambda\epsilon$, $N_{r,
R}(\gamma') \intersection \Ball_\rho(\gamma(a_0))$ is equal to
$g_i^{-1} N_{r, R}(\gamma') \intersection \Ball_\rho(\gamma(a_i))$ for all $i$.
This set contains at most $V$ vertices, so there exist $0 \leq i < j \leq 2^V$
such that
\begin{align*}
g_i^{-1}\mathcal{P}'\restricted{N_{r, R}(\gamma') \intersection
  \Ball_\rho(\gamma(a_i))} = g_j^{-1}\mathcal{P}'\restricted{N_{r, R}(\gamma')
  \intersection \Ball_\rho(\gamma(a_j))}.
\end{align*}

Let $a = a_i$ and $b = a_j$ and let $\gamma = \gamma'\restricted{[a - \eta, b +
\eta]}$. We claim that $\gamma$ is then a feature of the first kind described
in the proposition. Setting $g = g_jg_i^{-1}$ and $\mathcal{P} =
\mathcal{P}'\restricted{N_{r, R}(\gamma') \intersection
N_R(\gamma\restricted{[a, b]})}$, conditions \ref{cutpairshort},
\ref{cutpairdepth}, \ref{cutpairlocal} and~\ref{cutpairpartition} are
satisfied by definition of the $a_i$. Let $c \in [a, b]$ such that $\gamma'(c) \in X_0$. Then
$C_K(\gamma) \intersection \Ball_T(\gamma(c))$ is equal to $C_K(\gamma')
\intersection \Ball_T(\gamma'(c))$ since $\eta \geq \lambda(T+K) +
\lambda\epsilon$ and Lemma~\ref{lem:Adisconnectedlocally} guarantees that
both sets in $\mathcal{P}'$ meet this set, so condition~\ref{cutpairnontrivial}
is satisfied, too.\end{proof}

The existence of a geodesic segment with the properties described in the statement of
Proposition~\ref{prop:cutpairsfeature} can be checked by looking at just a
finite ball in $X$. Such a ball can be computed from a solution to the word
problem in $\Gamma$, which exists since $\Gamma$ is hyperbolic. Therefore we
immediately obtain the following corollary:

\begin{cor}\label{cor:cutpaircomputable} There is an algorithm that takes as
input a presentation for a hyperbolic group $\Gamma$ with generating set $S$,
a list of subsets of $S$ generating a collection $\mathcal{H}$ of maximal virtually cyclic
subgroups of $\Gamma$ and integers $\delta$, $\delta_\mathcal{H}$ and $n$ such that the
cusped space associated to $(\Gamma, \mathcal{H}, S)$ is $\delta$-hyperbolic
and satisfies $\ddag_n$ and such that the Cayley graph of each element of
$\mathcal{H}$ with respect to its given generating set is
$\delta_\mathcal{H}$-hyperbolic and returns the answer to the question ``does
$\boundary(\Gamma, \mathcal{H})$ contain a cut pair?'' \qed \end{cor}

\subsection{Non-cut pairs}

By a similar argument, we now show that there is an algorithm that determines
whether or not $\boundary X$ contains a non-cut pair. For the detection of
non-cut pairs we will need the following lemma.

\begin{lem}\label{lem:equivalenceclass} Suppose that $\boundary X$ does not
contain a cut point. Let $(x_n)_{n \in \integers}$ be a sequence of points in
$X$ with $x_n \to x_{\pm \infty}$ as $n\to \pm \infty$. Suppose that each
pair $\{x_n, x_{n+1}\}$ is a cut pair. Then so is $\{x_{-\infty},
x_\infty\}$.  \end{lem}

\begin{proof} $\boundary X$ is locally connected by the main theorem
of~\cite{bowditch99b}. $\boundary X$ is assumed not to contain a cut point,
so the results of sections 2 and~3 of~\cite{bowditch98} can be applied. We
recall some definitions from that paper. For $x \in \boundary X$ we define
$\val(x) \in \naturals$ to be the number of ends of $\boundary X - \{x\}$.
Then we let $M(n) = \{x \in \boundary X \colon \val(x) = n\}$ and $M(n{+}) =
\{x \in \boundary X \colon \val(x) \geq n\}$. For $x$ and $y$ in $M(2)$ we
write $x \sim y$ if $x = y$ or $\{x, y\}$ is a cut pair; this defines an
equivalence relation. For $x$ and $y$ in $M(3{+})$ we write $x \isom y$ if
$\val(x) = \val(y)$ and $\boundary X - \{x, y\}$ has exactly $\val(x)$
components.

Recall~\cite[Lemma 3.8]{bowditch98}: if $x \isom y$ and $x \isom z$ then $y
\isom z$. Therefore $x_n \in M(2)$ for all $n$, so $\{x_n\}_{n \in
\integers}$ is a subset of a $\sim$-equivalence class $\sigma$. By~\cite[Lem.\
3.2]{bowditch98} $\sigma$ is a cyclically separating set, and so is the closure
of $\sigma$ by~\cite[Lem.\ 2.2]{bowditch98}, which implies that $\{x_{-\infty},
x_\infty\}$ is a cut pair as required.  \end{proof}

Let $\lambda$,
$\epsilon$, $r$, $R$, $K$, $T$, $k$, $\rho$, $\eta$, $B$ and $V$ be as defined
in section~\ref{subsec:cutpairs}. Let $N_1$, $N_2$ and $N_3$ be given by
\begin{align*} 
  N_1 &= 2(V-1)((k+R+1)B^{2\eta}V^{V+1} + 2\eta) + 2\eta + 2((k+R+1)B^{2\eta} + 1),\\
  N_2 &= (k+R+1)B^{2\eta} + 1,\\
  N_3 &= 2(k + R + 1)B^{2\eta}V^{V+1} + 4\eta.
\end{align*}

\begin{prop}\label{prop:noncutpairfeature} $\boundary X$ contains a non-cut
pair if and only if $X$ contains one of the following two features:
\begin{enumerate}
  \item Geodesic segments $\gamma_i \colon [a_i - \eta, b_i + \eta] \to X$ with
    image in $X_k$ for $i = 1, 2, 3$ with $a_2 = b_1$ and $a_3 = b_2$ such that
    \begin{enumerate}
      \item\label{noncutpairshortout} $1 \leq b_i - a_i \leq N_1$ for $i = 1,
        3$,
      \item\label{noncutpairshortin} $1 \leq b_2 - a_2 \leq N_2$,
      \item\label{noncutpairlocal1} $\gamma_i\restricted{[b_i - \eta, b_i +
        \eta]} = \gamma_{i+1}\restricted{[a_i - \eta, a_i + \eta]}$ for $i = 1,
        2$,
      \item\label{noncutpairlocal2} $\gamma_i\restricted{[b_i - \eta, b_i +
        \eta]} = g_i \cdot\gamma_i\restricted{[a_i - \eta, b_i + \eta]}$ for $i
        = 1, 3$, and
      \item\label{noncutpairdepth} $h(\gamma_i(a_i)) = h(\gamma_i(b_i))$ for $i
        = 1, 3$, so there exist $g_i \in \Gamma$ such that $\gamma_i(b_i) =
        g_i\gamma(a_i)$,
      \item\label{noncutpairconnected} all vertices of $C_K(\gamma_2)
        \intersection \Ball_T(\gamma_2(c))$ lie in the same connected component of
        $N_{r, R}(\gamma_2) \intersection N_R(\gamma_2\restricted{[a_2, b_2]})$
        for some $c \in [a_2, b_2]$ such that $\gamma_2(c) \in X_0$.
    \end{enumerate}
  \item A geodesic segment $\gamma \colon [a, b] \to X$ with image in $X_k$ such that
  \begin{enumerate}
    \item $b - a \leq N_3$,
    \item $h(\gamma(a)) = h(\gamma(b)) = k$ with $\gamma$ descending vertically
      at $a$ and ascending vertically at $b$, and
    \item $A'_{r, R, K}(\gamma)$ is connected.
  \end{enumerate}
\end{enumerate}
\end{prop}

\begin{proof} First suppose that $X$ contains a feature of the first kind
described in the proposition.  Then define a path $\gamma'$ in $X$ as follows:
\begin{align*} \gamma'(t) = 
  \begin{dcases}
    g_1^m \cdot \gamma_1(t') & \text{if $t = m(b_1 - a_1) + t'$ for 
      $m \in \integers_{\leq 0}$ and $t' \in [a_1, b_1]$ }\\
    \gamma_2(t) & \text{if $t \in [a_2, b_2]$}\\
    g_3^m \cdot \gamma_3(t') & \text{if $t = m(b_3 - a_3) + t'$ for 
      $m \in \integers_{\geq 0}$ and $t' \in [a_3, b_3]$ }\\
  \end{dcases}
\end{align*}
That is, $\gamma'$ is obtained by concatenating infinitely many translates of
$\gamma_1$, then a copy of $\gamma_2$, then infinitely many translates of
$\gamma_3$. Note that this is an $(8\delta + 1)$-local-geodesic by
conditions~\ref{noncutpairlocal1} and~\ref{noncutpairlocal2} since $\eta \geq
8\delta + 1$ and is therefore a $(\lambda, \epsilon)$-quasi-geodesic.

$C_K(\gamma') \intersection \Ball_T(\gamma'(c))$ is equal to $C_K(\gamma_2)
\intersection \Ball_T(\gamma_2(c))$ since $\gamma'\restricted{[a_2 - \eta, b_2 +
\eta]} = \gamma_2$ and $\eta \geq \lambda(T + K) + \lambda\epsilon$.
Furthermore, $\eta \geq \lambda (R + r) + \lambda\epsilon$, which similarly
guarantees that $N_{r, R}(\gamma_2) \intersection N_R(\gamma_2\restricted{[a_2,
b_2]})$ is a subset of $N_{r, R}(\gamma')$. Therefore $C_K(\gamma')
\intersection \Ball_T(\gamma'(c))$ lies in a single component of $A_{r, R,
K}(\gamma')$ by condition~\ref{noncutpairconnected}, so $A_{r, R, K}(\gamma')$
is connected by Lemma~\ref{lem:Adisconnectedlocally}, so $\Lambda\gamma'$ is a
non-cut pair by the results of section~\ref{sec:annulus}.

Now suppose that a feature of the second type exists in $X$. Let $\hat\gamma$
be the path obtained by concatenating $\gamma$ with vertical geodesic rays.
Then $\hat\gamma$ is an $(8\delta+1)$-local-geodesic since $k \geq 8\delta+1$.
$A_{r, R, K}(\hat\gamma)$ is connected by Lemma~\ref{lem:geodesicsinhoroballs},
so $\Lambda\hat\gamma$ is a non-cut pair by the results of
section~\ref{sec:annulus}.

Conversely, suppose that $\boundary X$ contains a non-cut pair. Let $\gamma'$
be an $(8\delta+1)$-local-geodesic in $X$ such that $\Lambda\gamma'$ is such a
pair. Assume first that $\gamma'$ is contained in $X_{k+R}$ and reparametrise
$\gamma'$ so that $\gamma'(0) \in X_0$. $C_K(\gamma') \intersection
\Ball_T(\gamma'(0))$ contains at most $V$ vertices and lies in a single component
of $N_{r, R}(\gamma')$. Define $n_V^\pm$ to be $\pm\lambda(K+T+\epsilon)$ and
then for $l$ decreasing from $V-1$ to $1$ let $n_l^+$ and $n_l^-$ be chosen to
minimise $n_l^+ - n_l^-$ among pairs such that $C_K(\gamma') \intersection
\Ball_T(\gamma'(0))$ meets at most $l$ components of $N_{r, R}(\gamma')
\intersection N_R(\gamma'\restricted{[n_l^-, n_l^+]})$ and $[n_{l+1}^-,
n_{l+1}^+] \subset [n_l^-, n_l^+]$.  Note that the condition that
$\mod{n^\pm_l} \geq \lambda(K + T) + \lambda\epsilon$ ensures that $C_K(\gamma')
\intersection \Ball_T(\gamma'(0))$ is a subset of $N_{r, R}(\gamma') \intersection
N_R(\gamma'\restricted{[n_l^-, n_l^+]})$.

Suppose that $n_{l - 1}^+ - n_l^+ > (k + R + 1)B^{2\eta} V^{V+1} + 2\eta$. Let
$\mathcal{Q}$ be the partition of $C_K(\gamma') \intersection \Ball_T(\gamma'(0))$
into $l$ non-empty subsets induced by connectivity in $N_{r,
R}(\gamma') \intersection N_R(\gamma'\restricted{[n_{l-1}^-, n_{l-1}^+-1]})$.
For each $t \in [n_l + \eta, n_{l-1} - \eta]$ define the following sets:
\begin{align*}
Y_t &= N_{r, R}(\gamma') \intersection \Ball_\rho(\gamma'(t)) \\
Z_t &= N_{r, R}(\gamma') \intersection \left(\Ball_\rho(\gamma'(t)) \union
  N_R(\gamma'\restricted{[n_{l-1}^-, t]})\right)
\end{align*}
and let $\mathcal{P}_t$ be the partition of the vertices of $Y_t$ into $l + 1$
subsets: $l$ corresponding to the $l$ sets in $\mathcal{Q}$ by connectivity in
$Z_t$ and one containing the part of $Y_t$ not connected to $C_K(\gamma')
\intersection \Ball_T(\gamma'(0))$ in $Z_t$.  As in the proof of
Proposition~\ref{prop:cutpairsfeature} there exist $s_1 < s_2$ in $[n^+_l + \eta,
n^+_{l-1} - \eta]$ such that such that
\begin{enumerate}
\item $h(\gamma'(s_1)) = h(\gamma'(s_2))$, so $\gamma'(s_2) = g\gamma'(s_1)$
  for some $g$ in $\Gamma$,
\item $\gamma'\restricted{[s_2-\eta,s_2+\eta]} = g\cdot\gamma'
  \restricted{[s_1-\eta,s_1+\eta]}$, which implies that $Y_{s_2} = g\cdot
  Y_{s_1}$, and
\item 
  $\mathcal{P}_{s_2} = g\cdot\mathcal{P}_{s_1}$. 
\end{enumerate}
Then replace $\gamma'$ by another path defined by
\begin{align*} \gamma''(t) = 
  \begin{dcases}
    \gamma'(t) & \text{if $t \leq s_1$}\\
    g^{-1}\cdot \gamma'(t + (s_2 - s_1)) & \text{if $t \geq s_2$} \\
  \end{dcases}
\end{align*}

This is an $(8\delta+1)$-local-geodesic since $\eta \geq 8\delta+1$. By
definition of $\rho$, the intersection of $Z_{s_2}$ and $N_{r, R}(\gamma')
\intersection N_R(\gamma'\restricted{[s_2, n_{l-1}^+]})$ is contained in
$Y_{s_2}$. Then by definition of $n_{l-1}^+$, two distinct sets in
$\mathcal{P}_{s_2}$ meet $N_{r, R}(\gamma') \intersection
N_R(\gamma'\restricted{[s_2, n_{l-1}^+]})$. Since $\eta \geq \lambda(R + r +
\epsilon)$,
\begin{align*}
  g^{-1}N_{r, R}(\gamma') \intersection N_R(\gamma'\restricted{[s_2, n_{l-1}^+]})
  = N_{r, R}(\gamma'') \intersection N_R(\gamma'\restricted{[s_1, n_{l-1}^+ -
  (s_2 - s_1)]}),
\end{align*}
and it follows that $C_K(\gamma'') \intersection \Ball_T(\gamma'(0))$ meets at
most $l-1$ components of $N_{r, R}(\gamma'') \intersection
N_R(\gamma''\restricted{[n_{l-1}^-, n_{l-1}^+ - (s_2 - s_1)]})$. This implies
that the process of replacing $\gamma'$ by $\gamma''$ leaves unchanged
$n_{l'}^+$ for all $l' < l$ and $n_{l'}^-$ for all $l'$ and strictly reduces
$n_l^+$.  Therefore by repeating this process we can assume that $\gamma'$ was
chosen to ensure that $\mod{n_{l-1}^\pm - n_l^\pm}$ is at most
$(k+R+1)B^{2\eta}V^{V+1} + 2\eta$ for all $l$, and therefore that 
\begin{align*}
  \mod{n_1^\pm} \leq \lambda(K + T + \epsilon) + (V-1)((k+R+1)B^{2\eta}V^{V+1} + 2\eta).
\end{align*}

There exist $a_1 \leq b_1 \leq n_1^- - \eta$ with $n_1^- - b_1$ and $b_1 - a_1$
both at most $(k+R+1)B^{2\eta} + 1$ such that 
\begin{enumerate}
\item $h(\gamma'(b_1)) = h(\gamma'(a_1))$, so $\gamma'(b_1) = g_1\cdot
  \gamma'(a_1)$ for some $g_1 \in \Gamma$.
\item $\gamma'\restricted{[b_1 - \eta, b_1 + \eta]} =
  g_1\cdot\gamma'\restricted{[a_1 - \eta, a_1 + \eta]}$.
\end{enumerate}
Then let $\gamma_1 = \gamma'\restricted{[a_1 - \eta, b_1+\eta]}$. Similarly
define $b_3 \geq a_3 \geq n_1^+$ and let $\gamma_3 = \gamma'\restricted{[a_3 -
\eta, b_3 + \eta]}$. Let $a_2 = b_1$ and $b_2 = a_3$ and let $\gamma_2 =
\gamma'\restricted{[a_2 - \eta, b_2 + \eta]}$; note that $b_2 - a_2 \leq N_1$.
Then the triple $(\gamma_1, \gamma_2, \gamma_3)$ is a feature in $X$ of the
first kind listed in the proposition: conditions~\ref{noncutpairshortout},
\ref{noncutpairshortin}, \ref{noncutpairdepth}, \ref{noncutpairlocal1}
and~\ref{noncutpairlocal2} clearly hold by construction and
condition~\ref{noncutpairconnected} holds because the condition that $\eta \geq
\lambda(R + r) + \lambda\epsilon$ ensures that $N_{r, R}(\gamma') \intersection
N_R(\gamma'\restricted{[n_1^-, n_1^+]})$ is a subset of $N_{r, R}(\gamma_2)
\intersection N_R(\gamma_2\restricted{[a_2, b_2]})$.

If $\gamma'$ is not contained in $X_{k+R}$ let $\gamma'^{-1}h^{-1}[0, k+R]$ be
a (possibly infinite) union of (possibly infinite) intervals $\union_{i \in
I}[a_i-R, b_i+R]$ where we order the intervals so that $a_{i+1} > b_i$ for all
$i$. Then $h(\gamma'(a_i)) = h(\gamma'(b_i)) = k$ for all $i$ and we can apply
the results of section~\ref{sec:horoballs}.  For $i \in I$ let $\gamma'_i =
\gamma'\restricted{[a_i, b_i]}$. Let $\hat{\gamma}'_i$ be the bi-infinite
$(8\delta + 1)$-local-geodesic obtained by concatenating $\gamma_i$ with
either one or two vertical geodesic rays.

Suppose that $\Lambda\hat\gamma'_i$ is a cut pair for all $i$. Then
Lemma~\ref{lem:equivalenceclass} tells us that $\Lambda\gamma'$ is a cut pair,
which is a contradiction.  Therefore there exists $i$ such that
$\Lambda\hat\gamma_i$ is a non-cut pair. 

Arguing as before, the geodesic $\hat\gamma_i$ can be altered to ensure that
$b_i - a_i \leq 4\eta + 2(k+R+1)B^{2\eta}V^{V+1}$; this yields a feature of the
second type described in the proposition.  \end{proof}

From this we deduce the following corollary:

\begin{cor}\label{cor:noncutpaircomputable} There is an algorithm that takes as
input a presentation for a hyperbolic group $\Gamma$ with generating set $S$,
a list of subsets of $S$ generating a collection $\mathcal{H}$ of maximal virtually cyclic
subgroups of $\Gamma$ and integers $\delta$, $\delta_\mathcal{H}$ and $n$ such that the
cusped space associated to $(\Gamma, \mathcal{H}, S)$ is $\delta$-hyperbolic
and satisfies $\ddag_n$ and such that the Cayley graph of each element of
$\mathcal{H}$ with respect to its given generating set is
$\delta_\mathcal{H}$-hyperbolic and returns the answer to the question ``does
$\boundary(\Gamma, \mathcal{H})$ contain a non-cut pair?'' \qed \end{cor}

\section{Splittings of groups with circular boundary}\label{sec:fuchsiangroups}

The question of the existence of a relative splitting of a hyperbolic group
cannot be answered by consideration of the topology of its boundary alone if
the boundary is homeomorphic to a circle: some groups with circular boundary
split and some do not. In this section we deal with this special case. 

\subsection{Circular boundary}

Let $\Gamma$ by a hyperbolic group with a finite collection $\mathcal{H}$ of
subgroups such that $\Gamma$ is hyperbolic relative to $\mathcal{H}$ and
$\boundary(\Gamma, \mathcal{H})$ is homeomorphic to $S^1$. Then $\Gamma$ acts
as a discrete convergence group on $\boundary(\Gamma, \mathcal{H})$
by~\cite{bowditch99c}, so by the Convergence Group Theorem of Tukia, Casson,
Jungreis and Gabai~\cite{tukia88,cassonjungreis94,gabai92} there is a properly
discontinuous action of $\Gamma$ by isometries on $\mathbb{H}^2$ and a
$\Gamma$-equivariant homeomorphism $\boundary(\Gamma, \mathcal{H}) \to
\boundary\mathbb{H}^2$.

Let $K$ be the (finite) kernel of the action of $\Gamma$ on $\boundary(\Gamma,
\mathcal{H})$; note that this is the same as the kernel of the extension of the
action to $\mathbb{H}^2$. Then $\Gamma$ is an extension

  \begin{center}\begin{tikzpicture}[>=angle 90]
  \matrix (a) [matrix of nodes, row sep=3em, column sep = 2em, text height = 
    1.5ex, text depth=0.25ex]
  {
    $1$ & $K$ & $\Gamma$ & $\Gamma'$ & $1$\\
  };
  \path[->](a-1-1) edge (a-1-2)
    (a-1-2) edge (a-1-3)
    (a-1-3) edge (a-1-4)
    (a-1-4) edge (a-1-5);
  \end{tikzpicture}\end{center}
and the quotient $\Gamma'$ acts faithfully on $\mathbb{H}^2$. Let
$\mathcal{H}'$ be the set of images of elements of $\mathcal{H}$ in $\Gamma'$.

\begin{lem}\label{lem:quotientbyeffectivekernel} With $\Gamma$ and $\Gamma'$ as
above, $\Gamma$ splits non-trivially over a virtually cyclic subgroup relative to
$\mathcal{H}$ if and only if $\Gamma'$ admits such a splitting relative to
$\mathcal{H}'$.\end{lem}

\begin{proof} One direction is clear: if $\Gamma'$ acts minimally on a
non-trivial tree $T$ with virtually cyclic edge stabilisers then the quotient
map $\Gamma \to \Gamma'$ induces a minimal $\Gamma$-action on $T$ with
edge stabilisers finite extensions of the edge stabilisers of the
$\Gamma'$-action. 

Conversely, suppose that $\Gamma$ acts minimally on a non-trivial tree $T$ with
virtually cyclic edge stabilisers. $K$ is finite, so the restriction of the
action to $K$ fixes a vertex $v$. $K$ is normal, so its action fixes
$\Gamma\cdot v$ pointwise. Any point in $T$ lies on a geodesic path connecting
points in $\Gamma\cdot v$, since the union of all such paths is a
$\Gamma$-invariant subtree of $T$ and the action was assumed to be minimal.
Therefore $K$ acts trivially on $T$.

It follows that the $\Gamma$-action descends to a $\Gamma'$ action, and the
edge stabilisers are quotients of the original edge stabilisers by finite
subgroups.  Furthermore, elements of $\mathcal{H}'$ are elliptic if elements of
$\mathcal{H}$ are. \end{proof}

Any finite normal subgroup of $\Gamma$ is contained in the ball of radius
$4\delta + 2$ centred at the identity by~\cite{bridsonhaefliger99}, so by
checking all finite subsets of this ball using a solution to the word problem,
the set of finite normal subgroups of $\Gamma$ can be computed. $K$ is the
unique maximal finite normal subgroup of $\Gamma$ and is therefore computable. 

As described above, $\Gamma'$ can be realised as a discrete subgroup of
$\Isom\mathbb{H}^2$. The conjugacy classes of elements of $\mathcal{H}'$ are
then precisely the conjugacy classes of maximal parabolic subgroups of
$\Gamma'$. The action of $\Gamma'$ on $\boundary(\Gamma', \mathcal{H}')$ is
minimal, so the limit set of the action is $S^1$. It follows that $\Gamma'$ is
a Fuchsian group of the first kind: $\mathbb{H}^2/\Gamma'$ is a finite volume
orbifold and $\mathcal{H}'$ is a choice of conjugacy class representatives for
the cusp subgroups. Truncating the cusps of the orbifold, we realise $\Gamma'$
as the fundamental group of a compact hyperbolic orbifold such that
$\mathcal{H}'$ is a choice of conjugacy class representatives for the boundary
subgroups.  

\begin{defn}\label{defn:boundedFuchsian} A group is~\emph{bounded Fuchsian} if
it acts properly discontinuously and convex cocompactly (i.e.\ cocompactly on
a convex subset) on $\mathbb{H}^2$. Then a group is bounded Fuchsian if and
only if it surjects with finite kernel onto the fundamental group of a compact
two-dimensional hyperbolic orbifold. The \emph{peripheral
subgroups} of a bounded Fuchsian group are the fundamental groups of the
boundary components of the associated orbifold.\end{defn}

Therefore we have shown that $\Gamma'$ is a bounded Fuchsian group and
$\mathcal{H}'$ is a collection of representatives of its conjugacy classes of
peripheral subgroups. This proves the following lemma:

\begin{lem}\label{lem:grouptoorbifold} There is an algorithm that, when given a
presentation for a hyperbolic group with a set $\mathcal{H}$ of virtually
cyclic peripheral subgroups, returns a presentation for another hyperbolic
group $\Gamma'$ with a set $\mathcal{H}'$ of virtually cyclic subgroups such
that $\boundary(\Gamma', \mathcal{H}')$ is homeomorphic to $\boundary(\Gamma,
\mathcal{H})$ and $\Gamma$ splits over a virtually cyclic subgroup relative
to $\mathcal{H}$ if and only if $\Gamma'$ splits over a virtually cyclic
subgroup relative to $\mathcal{H}'$. Furthermore, if $\boundary(\Gamma',
\mathcal{H}')$ is homeomorphic to a circle then $\Gamma'$ is the fundamental
group of a compact two-dimensional hyperbolic orbifold and $\mathcal{H}'$ is
a set of representatives of fundamental groups of components of the boundary of
the orbifold.  \end{lem}

\subsection{Orbifolds}

Given the results of the previous section, the problem of determining whether
or not a given group $\Gamma$ with circular boundary splits non-trivially over
a virtually cyclic subgroup relative to a collection $\mathcal{H}$ of
peripheral subgroups reduces to the case in which $\Gamma$ is a bounded
Fuchsian group and $\mathcal{H}$ is a set of conjugacy class representatives of
its peripheral subgroups. Let $\Gamma = \pi_1Q$ where $Q$ is a compact
hyperbolic orbifold, so $\mathcal{H}$ is a set of conjugacy class
representatives of the boundary subgroups of $Q$.

We recall some terminology to describe orbifolds. The universal cover of a
hyperbolic orbifold (possibly with boundary) is a convex subspace $\tilde{Q}$
of $\mathbb{H}^2$. The orbifold fundamental group $\pi_1Q$ acts properly
discontinuously and by isometries on $\tilde{Q}$. The \emph{underlying
topological surface} $Q_\text{top}$ of $Q$ is defined to be the topological
quotient of $\tilde{Q}$ by this action. The \emph{topological boundary}
$\boundary_\text{top}$ of $Q$ is the boundary of $Q_\text{top}$, while the
\emph{orbifold boundary} $\boundary Q$ of $Q$ is the image of the boundary of
$\tilde{Q}$ in $Q_\text{top}$ under the quotient projection. A \emph{mirror} in
$Q$ is the image in $Q_\text{top}$ of the set of fixed points of an order 2
orientation reversing isometry of $\tilde{Q}$ in $\pi_1Q$. All mirrors in $Q$
are contained in $\boundary_\text{top}Q$ and any mirror is homeomorphic either
to a circle or to an interval. In the latter case each end point of the
interval is contained in either $\boundary Q$ or in another mirror. The
intersection of two mirrors is called a \emph{corner reflector}; the stabiliser
in $\pi_1Q$ of the preimage in $\tilde{Q}$ of a corner reflector is a finite
dihedral group. If $x$ is a point in $\tilde{Q}$ whose stabiliser in $\pi_1Q$
is non-trivial, cyclic and orientation preserving then the image of $x$ in
$Q_\text{top}$ is called a \emph{cone point}. The \emph{singular locus} of $Q$
is the union of all mirrors, corner reflectors and cone points. If $y$ is not
contained in the singular locus of $Q$ then any preimage of $y$ in $\tilde{Q}$
has trivial stabiliser in $\pi_1Q$.

A \emph{geodesic} in $Q$ is a curve $\gamma$ that is the image of a geodesic
$\tilde\gamma$ in $\tilde{Q}$ under the quotient projection. A geodesic
$\gamma$ is \emph{closed} if it is compact and \emph{simple} if it is closed
and furthermore $g\cdot\tilde\gamma$ is either equal to or disjoint from
$\tilde\gamma$ for each $g$ in $\pi_1Q$. A simple closed geodesic is
homeomorphic to either a circle or an interval with each of its end points
contained in a mirror in $Q$. A simple closed geodesic is \emph{essential} if
it is not contained in $\boundary_\text{top}Q$. For more information about
orbifolds see~\cite{scott83}.

The theory of splittings of fundamental groups of orbifolds relative to their
boundary subgroups is developed in~\cite{guirardellevitt16}; we recall the
following results:

\begin{lem}\cite[Corollary 5.6]{guirardellevitt16}\label{lem:splitiffnonsmall}
$\Gamma$ splits non-trivially relative to $\mathcal{H}$ if and only if $Q$
contains an essential simple closed geodesic.\end{lem}

\begin{defn} We call a compact orbifold without an essential simple closed geodesic 
\emph{small}.\end{defn}

\begin{prop}\cite[Proposition 5.12]{guirardellevitt16}\label{prop:characterisationofsmall}
A hyperbolic 2-orbifold $Q$ is small if and only if it is one of the following. 
\begin{enumerate}
  \item A sphere with three cone points, so $\pi_1Q \isom \langle a, b \vert
    a^p, b^q, ab^r\rangle$ where $p^{-1}+q^{-1}+r^{-1} < 1$ and the peripheral
    structure is empty.
  \item A triangle, all three edges of which are mirrors, so $\pi_1Q$ has
    presentation $\langle a, b, c \vert a^2, b^2, c^2, (ab)^p, (bc)^q,
    (ca)^r\rangle$ where $p^{-1}+q^{-1}+r^{-1} < 1$ and the peripheral
    structure is empty.
  \item A disc with two cone points, so $\pi_1Q \isom \langle a, b \vert a^p,
    b^q\rangle$ where $p, q > 1$ and the peripheral structure is $\{\langle
    ab\rangle\}$.
  \item A cylinder with one cone point, so $\pi_1Q \isom \langle a, b \vert
    (ab)^p\rangle$ where $p >1$ and the peripheral structure is $\{\langle
    a\rangle, \langle b\rangle\}$.
  \item A pair of pants, so $\pi_1Q$ is the free group $\langle a, b
    \vert\rangle$ and the peripheral structure is $\{\langle a\rangle, \langle
    b\rangle, \langle c\rangle\}$.
  \item A disc with one cone point with edge consisting of an interval boundary
    component and a mirror, so $\pi_1Q \isom \langle a, t \vert a^2,
    t^p\rangle$ where $p > 1$ and the peripheral structure is $\{\langle a,
    tat^{-1}\rangle\}$.
  \item A square with an interval boundary component and three mirrors edges,
    so $\pi_1Q \isom \langle a, b, c\vert a^2, b^2, c^2, (ab)^p, (bc)^q\rangle$
    where $p + q \geq 1$ and the peripheral structure is $\{\langle a,
    c\rangle\}$.
  \item An annulus in which one edge comprises an interval boundary component and a
    mirror and the other is a circular boundary component, so $\pi_1Q \isom
    \langle a, t \vert a^2\rangle$ with peripheral structure $\{\langle a,
    tat^{-1}\rangle\}$.
  \item A pentagon with two non-adjacent interval boundary components and three
    mirrors as edges, so $\pi_1Q \isom \langle a, b, c \vert a^2, b^2, c^2,
    (ab)^p\rangle$ with peripheral structure $\{\langle b, c\rangle, \langle c,
    a\rangle\}$.
  \item A hexagon, the six edges of which are alternately interval boundary
    components and mirrors, so $\pi_1Q \isom \langle a, b, c \vert a^2, b^2,
    c^2\rangle$ and the peripheral structure is $\{\langle a, b\rangle, \langle
    b, c\rangle, \langle c, a\rangle\}$.
\end{enumerate}
\end{prop}

\begin{lem}\label{lem:identifyingsmallorbifolds} There is an algorithm that
  takes as input a presentation for a hyperbolic group $\Gamma$ and a
  collection $\mathcal{H}$ of maximal virtually cyclic peripheral subgroups and
  terminates if and only if $\boundary(\Gamma, \mathcal{H})$ is homeomorphic to
  a circle and $\Gamma$ does not split relative to $\mathcal{H}$ over a
  virtually cyclic subgroup.
\end{lem}

\begin{proof} First use the algorithm of Lemma~\ref{lem:grouptoorbifold} and
let the output of that algorithm be $(\Gamma', \mathcal{H}')$; then
$\boundary(\Gamma, \mathcal{H})$ is homeomorphic to a circle and $\Gamma$
does not split over a virtually cyclic subgroup relative to $\mathcal{H}$ if
and only if there is an isomorphism from $\Gamma'$ to one of the groups with
peripheral structure listed in proposition~\ref{prop:characterisationofsmall}
that maps elements of $\mathcal{H}'$ to conjugates of elements of that
peripheral structure.
  
Enumerate the groups and peripheral structures described in
Proposition~\ref{prop:characterisationofsmall} and, in parallel, enumerate
all homomorphisms from these groups to $\Gamma'$ and homomorphisms from
$\Gamma'$ to these groups. Note that this is possible since one can test
whether or not a map defined on the generators of a group extends to a
homomorphism using a solution to the word problem in the codomain of the map,
and $\Gamma'$ and all groups listed in
Proposition~\ref{prop:characterisationofsmall} are hyperbolic. The algorithm
then terminates when an inverse pair of such maps that preserve the peripheral
structures (up to conjugacy) is found.  \end{proof}

\section{Maximal splittings}\label{sec:maximal}

We now apply the technical results of sections~\ref{sec:cutpairsalgorithms}
and~\ref{sec:fuchsiangroups} to the problem of finding a maximal splitting of a
one-ended hyperbolic group.

\subsection{JSJ decompositions}

We begin by recalling the definition a JSJ decomposition. For a more detailed
account see~\cite{guirardellevitt16}. Let $\Gamma$ be a group and let
$\mathcal{A}$ be a collection of subgroups of $\Gamma$ that is closed under
conjugation and taking subgroups. Let $\mathcal{H}$ be another collection of
subgroups of $\Gamma$. An \emph{$(\mathcal{A}, \mathcal{H})$-tree} is a tree
with a $\Gamma$-action with no edge inversions in which the stabiliser of each
edge is in $\mathcal{A}$ and each element of $\mathcal{H}$ is \emph{elliptic},
that is, it fixes a point in the tree. An $(\mathcal{A}, \mathcal{H})$-tree $T$
\emph{dominates} another such tree $T'$ if there is a $\Gamma$-equivariant
morphism from $T$ to $T'$.  $T$ is \emph{elliptic with respect to $T'$} if the
stabiliser of each edge in $T$ is elliptic in $T'$. For $v$ a vertex of $T$ we
denote by $\Gamma_v$ the stabiliser of $v$, and similarly if $e$ is an edge of
$T$ we let $\Gamma_e$ denote the stabiliser of $e$. We will sometimes refer to
a $(\mathcal{A}, \mathcal{H})$-tree as a \emph{splitting of $\Gamma$ over
$\mathcal{A}$ relative to $\mathcal{H}$}. A splitting is \emph{trivial} if
$\Gamma$ is elliptic. We will frequently implicitly move between the languages
of group actions on trees and graphs of groups.

Fixing $\mathcal{A}$ and $\mathcal{H}$, a subgroup of $\Gamma$ is
\emph{universally elliptic} if it is elliptic with respect to any
$(\mathcal{A}, \mathcal{H})$-tree. An $(\mathcal{A}, \mathcal{H})$-tree is
universally elliptic if the stabiliser of each if its edges is universally
elliptic. An $(\mathcal{A}, \mathcal{H})$ tree is a \emph{JSJ tree} if it is
universally elliptic and is maximal for domination among universally elliptic
$(\mathcal{A}, \mathcal{H})$ trees. 

\subsection{The Bowditch JSJ}

In~\cite{bowditch98} Bowditch defines a canonical JSJ tree $\Sigma$ for a
one-ended hyperbolic group $\Gamma$ where $\mathcal{A}$ is the set of virtually
cyclic subgroups of $\Gamma$ and $\mathcal{H}$ is empty. (To ensure that
$\mathcal{A}$ is closed under taking subgroups we should strictly speaking also
include the finite subgroups of $\Gamma$ in $\mathcal{A}$. However, since
$\Gamma$ is assumed to be one-ended it does not split over any finite subgroup,
so this change is not important.) We recall a description of the tree $\Sigma$ here.

$\Sigma$ is a tree with a three-colouring: its vertex set $V(\Sigma)$ admits a
partition $V_1(\Sigma) \disjointunion V_2(\Sigma) \disjointunion V_3(\Sigma)$
preserved by the action of $\Gamma$ such that no two vertices in either
$V_1(\Sigma)$ or $V_2(\Sigma) \union V_3(\Sigma)$ are adjacent.

The stabiliser of a vertex in $V_1(\Sigma)$ is a maximal virtually cyclic subgroup of
$\Gamma$ and therefore contains the stabiliser of each incident edge at finite
index.

To describe the vertices of the second type we require the following
definition. Recall definition~\ref{defn:boundedFuchsian} of a bounded Fuchsian
group.

\begin{defn} A \emph{hanging Fuchsian} subgroup $Q$ of $\Gamma$ is a subgroup
of $\Gamma$ that is the stabiliser of a vertex in some finite splitting of
$\Gamma$ over virtually cyclic subgroups such that $Q$ admits an isomorphism
with a bounded Fuchsian group that maps the stabilisers of incident
edges precisely to the peripheral subgroups.\end{defn}

Stabilisers of vertices in $V_2(\Sigma)$ are precisely the maximal hanging
Fuchsian subgroups of $\Gamma$ and the stabilisers of incident edges at such a
vertex in the JSJ decomposition are precisely the peripheral subgroups referred
to in the definition.

The stabiliser of a vertex in $V_3(\Sigma)$ is not virtually cyclic and is not
a hanging Fuchsian subgroup.

\subsection{Splittings and the topology of the boundary}\label{sec:splittingstopology}

In this section we recall and extend some results linking the existence of a
non-trivial splitting of a hyperbolic group relative to a collection of
virtually cyclic subgroups to the topology of a particular Bowditch boundary. 

For the remainder of Section~\ref{sec:splittingstopology}, we fix a hyperbolic
group $\Gamma$ and a collection $\mathcal{H}$ of virtually cyclic subgroups. If
$H$ is any virtually cyclic subgroup of $\Gamma$, let $\widehat{H}$ be the unique
maximal virtually cyclic subgroup of $\Gamma$ containing $H$. Then let
$\widehat{\mathcal{H}}$ be $\{\widehat{H} \vert H \in \mathcal{H}\}$. Recall
Lemma~\ref{lem:vcycperipheral}: since each group in $\widehat{\mathcal{H}}$ is
maximal virtually cyclic, $\Gamma$ is hyperbolic relative to
$\widehat{\mathcal{H}}$. We may therefore study the boundary $\boundary(\Gamma,
\widehat{\mathcal{H}})$.

\subsubsection{Boundaries with cut points}

We first recall results that deal with the case in which $\boundary(\Gamma,
\widehat{\mathcal{H}})$ contains a cut point. First recall
Bowditch's definition~\cite{bowditch01} of a \emph{peripheral splitting}.

\begin{defn} If $\Gamma$ is a group and $\mathcal{H}$ is a finite collection of
subgroups of $\Gamma$, a \emph{peripheral splitting} of $\Gamma$ with respect
to $\mathcal{H}$ is a representation of $\Gamma$ is a finite bipartite graph of
groups such that each vertex group of one colour is conjugate to an element of
$\mathcal{H}$, and each element of $\mathcal{H}$ is conjugate to a vertex group
of that colour.

As for any splitting, we say that a peripheral splitting is \emph{trivial} if
some vertex group is equal to $\Gamma$.\end{defn}

Note that, in our setting, a peripheral splitting of $\Gamma$ with respect to
$\widehat{\mathcal{H}}$ is a splitting of $\Gamma$ over virtually cyclic subgroups
relative to $\mathcal{H}$.

Then the following proposition, which we obtain by putting together two
theorems of Bowditch, completes our treatment of the case in which the boundary
contains a cut point.

\begin{prop}\label{prop:cutpointimpliessplitting} Suppose that
  $\boundary(\Gamma, \widehat{\mathcal{H}})$ is connected and contains a cut point.
Then the pair $\Gamma$ admits a non-trivial splitting over virtually cyclic
subgroups relative to $\mathcal{H}$.
\end{prop}

\begin{proof} By~\cite[Theorem 0.2]{bowditch99a} the global cut
  point of $\boundary(\Gamma, \widehat{\mathcal{H}})$ is a parabolic fixed point.
By~\cite[Theorem 1.2]{bowditch99b} $\Gamma$ admits a non-trivial peripheral
splitting with respect to $\widehat{\mathcal{H}})$. The edge groups in this splitting
are automatically virtually cyclic, since each edge meets a vertex with vertex
group conjugate into $\widehat{\mathcal{H}}$, and the splitting is automatically
relative to $\mathcal{H}$: in fact it is relative to $\widehat{\mathcal{H}}$.
\end{proof}

\subsubsection{Boundaries with cut pairs}

In the absence of cut points, the existence of a relative splitting is
reflected in the existence of cut pairs in the boundary.  In the absolute case,
recall the following theorem of~\cite{bowditch98}.

\begin{thm}\cite[Theorem 6.2]{bowditch98}\label{thm:cutpairimpliessplitting}
Let $\Gamma$ be a one-ended hyperbolic group such that $\boundary(\Gamma,
\emptyset)$ contains a cut pair and is not homeomorphic to $S^1$. Then $\Gamma$
admits a non-trivial splitting over a virtually cyclic subgroup.  \end{thm}

We require a relative version of this theorem. We only require such a result in the case
when $\Gamma$ arises as a vertex group in a splitting of a larger group over
virtually cyclic subgroups, and $\mathcal{H}$ is the collection of edge groups
incident at that vertex. In this section we show that in this simple case the
relative result follows from Theorem~\ref{thm:cutpairimpliessplitting}, and one
can avoid $\reals$-trees machinery. For a discussion in greater generality,
see~\cite{groff13}.

\begin{prop}\label{prop:cutpairimpliesrelativesplitting} Let $v$ be a vertex in
a minimal $\Gamma$-tree $T$ with virtually cyclic edge groups where $\Gamma$
is hyperbolic and one-ended. Let $\inc v$ be a set of representatives of
$\Gamma_v$-conjugacy classes of stabilisers of edges in $T$ incident at $v$.
Suppose that $\boundary(\Gamma_v, \inc v)$ is not a single point, does not
contain a cut point and is not homeomorphic to a circle but does contain a
cut pair. Then $\Gamma_v$ admits a non-trivial splitting over virtually cyclic
subgroups of $\Gamma_v$ relative to $\inc v$.  \end{prop}

First we need the following lemma. Recall that for a vertex $v$ in a
$\Gamma$-tree, we defined $\inc v$ to be a set of conjugacy class
representatives of the stabilisers of the edges of the $\Gamma$-tree incident
at $v$, and let $\widehat{\inc v}$ be the set of maximal virtually cyclic
subgroups of $\Gamma_v$ that contain the elements of $\inc v$.

\begin{lem}\label{lem:fixedcomponent} Let $f \colon T_1 \to T_2$ be an
  equivariant map of $\Gamma$-trees with virtually infinite cyclic edge
  stabilisers such that the action of $\Gamma$ on $T_1$ is cocompact and the
  action of $\Gamma$ on $T_2$ is minimal.  (This means that there is no proper
  $\Gamma$-invariant subtree of $T_2$.) Let $v$ be a vertex of $T_2$ such that
  $\boundary(\Gamma_v, \widehat{\inc v})$ is not a single point, is connected
  and does not contain a cut point. Then the action of $\Gamma_v$ on $T_1$
  fixes a component of $f^{-1}(v)$.\end{lem}

\begin{proof} First we show that there is a vertex $w \in T_1$ such that
$\Gamma_w \intersection \Gamma_v$ is non-elementary. If this is not the case
then the stabiliser of each edge of $T_1$ with respect to the action of
$\Gamma_v$ on $T_1$ is either finite or commensurable with the stabilisers of
its end points. Therefore the action induces a splitting of $\Gamma_v$ with
finite edge groups and virtually cyclic vertex groups. By minimality of the
action of $\Gamma$ on $T_2$ each edge $e$ incident at $v$ is $f(e')$ for some
edge $e'$ of $T_1$, and then $\Gamma_{e'}$ is a finite index subgroup of
$\Gamma_e$, since each is virtually infinite cyclic. In particular,
$\widehat{\Gamma_e}$ is elliptic, and the splitting is relative to $\widehat{\inc v}$.
But $\boundary(\Gamma_v, \widehat{\inc v})$ was assumed to be connected and not a
single point, which is a contradiction.

Then $f(w) = v$, otherwise any edge separating $f(w)$ from $v$ in $T_2$ has
non-elementary stabiliser. Let $S$ be the component of $f^{-1}(v)$ containing
$w$. We now show that any other vertex $w'$ of $T_1$ such that $\Gamma_{w'}
\intersection \Gamma_v$ is non-elementary is also in $S$. Suppose that $e$ is
an edge of $T_1$ that is not in $f^{-1}(v)$. As in Section~1
of~\cite{bowditch98} there exists a partition of $\boundary(\Gamma, \emptyset) -
\Lambda\Gamma_e$ as $U_1 \disjointunion U_2$.  The intersection of $\Gamma_e$
with $\Gamma_v$ is either finite or commensurable with a conjugate of an
element of $\widehat{\inc v}$, so the images of $U_1 \intersection
\Lambda\Gamma_v$ and $U_2 \intersection \Lambda\Gamma_v$ under the projection
map $\Lambda\Gamma_v \to \boundary(\Gamma, \widehat{\inc v})$ cover all but at
most a point of $\boundary(\Gamma, \widehat{\inc v})$. These sets are disjoint,
so one must be empty, say $U_2$. But $\Lambda\Gamma_w$ and $\Lambda\Gamma_{w'}$
each contain more than two points, so must both be contained in $U_1 \union
\Lambda\Gamma_e$. This implies that $w$ and $w'$ are on the same side of $e$.

Therefore the action of $\Gamma_v$ on $T_1$ fixes $S$, for any element of
$\Gamma_v$ must send $w$ to a vertex of $S$.\end{proof}

\begin{proof}[Proof of Proposition~\ref{prop:cutpairimpliesrelativesplitting}]
Let $\Sigma$ be Bowditch's JSJ tree for $\Gamma$. $\Sigma$ is
then elliptic with respect to $T$, so by~\cite[Proposition
2.2]{guirardellevitt16} there exists a $\Gamma$-tree $\widehat\Sigma$ and maps $p
\colon \widehat\Sigma \to \Sigma$ and $f \colon \widehat\Sigma \to T$ such that:
\begin{enumerate}
  \item $p$ is a collapse map. (i.e.\ a map given by collapsing some edges of
    $\widehat\Sigma$ to points.)
  \item For $w \in \Sigma$, the restriction of $f$ to $p^{-1}(w)$ is injective.
\end{enumerate}

Let $S \subset \widehat\Sigma$ be the component of $f^{-1}(v)$ fixed by the action
of $\Gamma_v$ constructed in Lemma~\ref{lem:fixedcomponent}. Suppose that a
vertex $w$ of $S$ is fixed by the $\Gamma_v$-action.

If $e$ is any edge of $\widehat\Sigma$ that is adjacent to but not contained in $S$
then $\Gamma_e \leq \Gamma_{f(e)}$ and the subgroup is necessarily of finite
index. Conversely the stabiliser of any edge incident at $v$ contains the
stabiliser of an edge adjacent to $S$ at finite index. The stabiliser of any
edge adjacent to $S$ is then commensurable with the stabiliser of an edge
incident at $w$, and vice versa. Therefore $\boundary(\Gamma_v, \widehat{\inc w})$ is
homeomorphic to $\boundary(\Gamma_v, \widehat{\inc v})$. We assumed that
$\boundary(\Gamma_v, \widehat{\inc v})$ was neither a point nor homeomorphic to a circle,
so $pw$ is not in $V_1(\Sigma)$ or $V_2(\Sigma)$, and is therefore in
$V_3(\Sigma)$.

Let $x$ and $y$ be points in $\boundary(\Gamma_v, \widehat{\inc v})$ and choose preimages
$\widetilde{x}$ and $\widetilde{y}$ in $\boundary(\Gamma_v, \emptyset)$, which
we identify with $\Lambda\Gamma_v \subset \boundary(\Gamma, \emptyset)$.  
The set of components of $\boundary(\Gamma_v, \widehat{\inc v}) - \{x, y\}$ is in
bijection with the set of those components of $\boundary(\Gamma, \emptyset) -
\{\widetilde{x}, \widetilde{y}\}$ that meet $\Lambda\Gamma_v$. 

Suppose then
that $\{x, y\}$ is a cut pair in $\boundary(\Gamma_v, \widehat{\inc v})$, so at least two
components of $\boundary(\Gamma, \emptyset) - \{\widetilde{x}, \widetilde{y}\}$ meet
$\Lambda\Gamma_v$.

Then by Theorem~\ref{thm:cutpairimpliessplitting} there is a type 1 or type 2
vertex $u$ of $\Sigma$ such that $\Lambda\Gamma_u$ contains $\{\widetilde{x},
\widetilde{y}\}$. Then $\{\widetilde{x}, \widetilde{y}\} = \Lambda\Gamma_u
\intersection \Lambda\Gamma_v$, so $\{\widetilde{x}, \widetilde{y}\}$ is the
limit set of an edge incident at $v$. Therefore $\{x, y\} \subset
\boundary(\Gamma_v, \widehat{\inc v})$ is a single point, which is a contradiction
because $\boundary(\Gamma_v, \widehat{\inc v})$ was assumed not to contain a cut
point. Hence the action of $\Gamma_v$ on $S$ does not fix any vertex and
therefore gives rise to a non-trivial splitting of $\Gamma_v$ relative to $\inc v$.

\end{proof}

\subsubsection{Boundaries without cut points or pairs}

Our description of the relationship between the existence of splittings and the
topology of the boundary is completed by the following proposition, which serves as a
converse to Propositions~\ref{prop:cutpointimpliessplitting}
and~\ref{prop:cutpairimpliesrelativesplitting}.

\begin{prop}\label{prop:relativesplittingimpliescutpair} Let $\Gamma$ be a
  hyperbolic group and let $\mathcal{H}$ be a finite set of virtually cyclic
  subgroups of $\Gamma$ such that $\boundary(\Gamma, \widehat{\mathcal{H}})$ is
    connected.  Suppose that $\Gamma$ admits a non-trivial splitting over a
  virtually cyclic subgroup relative to $\mathcal{H}$.  Then $\boundary(\Gamma,
\widehat{\mathcal{H}})$ contains a cut point or pair.\end{prop}

\begin{proof} Let $T$ be the $\Gamma$-tree associated to such a non-trivial splitting.
Without loss of generality assume that the action of $\Gamma$ on $T$ is
minimal. Let $e$ be any edge in $T$. Let $T_1$ and $T_2$ be the two
components of the complement of the interior of $e$ in $T$. Then as in the proof of Lemma~\ref{lem:fixedcomponent}
 we obtain a partition of $\boundary(\Gamma,
\emptyset) - \Lambda\Gamma_e$ as $U_1\disjointunion U_2$ where $U_i$ are open
sets given by
\begin{align*}
  U_i = \boundary T_i \union \bigunion_{w \in T_i} (\Lambda\Gamma_w - \Lambda\Gamma_e)
\end{align*}

If a subgroup $H$ of $\Gamma$ is in $\mathcal{H}$, $H \leq \Gamma_w$ for some vertex $w \in \Sigma$ and either $\Lambda H =
\Lambda\Gamma_e$ or $\Lambda H \intersection \Lambda\Gamma_e = \emptyset$.
In the latter case either
$\Lambda H \subset U_1$ or $\Lambda H \subset U_2$. Let $\pi \colon
\boundary(\Gamma, \emptyset) \to \boundary(\Gamma, \widehat{\mathcal{H}})$ be the
quotient projection of Lemma~\ref{lem:bowditchfromgromov}. It follows that the images
of $U_1$ and $U_2$ under $\pi$ in the complement of $\pi(\Lambda\Gamma_e)$ in
$\boundary(\Gamma, \widehat{\mathcal{H}})$ are disjoint open sets; they are non-empty by
the minimality of the action of $\Gamma$ on $T$. The image of $\Lambda\Gamma_e$
is either one or two points, and therefore $\boundary(\Gamma, \widehat{\mathcal{H}})$
contains a cut point or pair. \end{proof}

\subsection{Virtually cyclic subgroups and finding splittings}

We will need the following lemma, which allows us to do various computations
related to virtually cyclic subgroups of hyperbolic groups.

\begin{lem}\cite[Lemma 2.8]{dahmaniguirardel11}\label{lem:twoendedsubgroups}
There is an algorithm that, when given a presentation for a hyperbolic group
$\Gamma$ and a finite subset $S \subset \Gamma$, returns an answer to the
question ``is $\langle S \rangle \leq \Gamma$ virtually cyclic?'' If the
answer is ``yes'' then the algorithm also determines
\begin{enumerate}
\item the (unique) maximal finite normal subgroup of $\langle S \rangle$, 
\item a presentation for $\langle S\rangle$,
\item whether $\langle S \rangle$ is of type $\mathcal{Z}$ or $D_\infty$.
  (Recall that we say that a virtually cyclic group of type $\mathcal{Z}$
  (respectively $D_\infty$) if it surjects onto $\integers$ (respectively
  $D_\infty$).)
\item a generating set for the maximal virtually cyclic subgroup of $\Gamma$
  containing $\langle S\rangle$.\end{enumerate}
\end{lem}

The proof of this lemma in~\cite{dahmaniguirardel11} uses Makanin's algorithm
for solving equations in hyperbolic groups. We modify that part of the argument
to use only elementary methods in keeping with the themes of this paper.

\begin{proof} We give an alternative method to determine whether or not
$\langle S\rangle$ is virtually cyclic and to produce a maximal finite normal
subgroup of $\langle S\rangle$ in the case that it is; the rest of the
argument can be copied verbatim from~\cite{dahmaniguirardel11}. First compute
$\delta$ with respect to which $\Gamma$ is $\delta$-hyperbolic.

Use the algorithm of~\cite[Proposition 4]{kapovich96} to search for a constant
$K$ with respect to which $\langle S\rangle$ is $K$-quasi-convex in $\Gamma$.
This algorithm finds such a constant if it exists and does not terminate of
$\langle S\rangle$ is not quasi-convex; note that if $\langle S\rangle$ is
virtually cyclic then it is guaranteed to be quasi-convex. If the algorithm
terminates use $K$ and $\delta$ to compute $\delta'$ such that $\langle
S\rangle$ is $\delta'$-hyperbolic.  Then all finite subgroups of $\langle
S\rangle$ can be conjugated into a ball of radius at most $4\delta' + 2$ with
respect to the word metric in $\langle S\rangle$, so all finite normal
subgroups of $S$ can be computed using a solution to the word and conjugacy
problems in $\Gamma$. Once this is computed the algorithm
of~\cite{dahmaniguirardel11} can be used to determine whether or not $\langle
S\rangle$ is virtually cyclic.

In parallel, search for a pair of elements $g$ and $h$ in $\langle S\rangle$
such that $[g^2, h^2]$ has infinite order. This can be checked since the order
of an element of $\Gamma$ of finite order is bounded above by the number of
elements of $\Gamma$ in the ball of radius $4\delta + 2$.

If $\langle S\rangle$ is not quasi-convex then it contains a free group on two
generators, so a pair $(g, h)$ as in the previous paragraph certainly exists.
Conversely, if such a pair exists then $\langle S\rangle$ cannot be virtually cyclic,
since any virtually cyclic group contains a subgroup of index two that surjects onto
$\integers$ with finite kernel.  \end{proof}

\begin{lem}\label{lem:findingasplitting} There is an algorithm that takes as
  input a presentation for a hyperbolic group $\Gamma$ and a collection
  $\mathcal{H}$ of peripheral subgroups and either returns the graph of groups
  associated to a non-trivial splitting of $\Gamma$ relative to $\mathcal{H}$
over virtually cyclic subgroups or does not terminate if no such splitting
exists.  \end{lem}

\begin{proof} Suppose that $\Gamma$ admits a proper splitting relative to
  $\mathcal{H}$ as an amalgamated product over a virtually cyclic subgroup; the
  case of an HNN extension is similar. Then $\Gamma$ admits a presentation that
  makes this splitting explicit in the following sense. There are finite
  disjoint symmetric sets of symbols $S_1$, $S_2$ and $S_3$, finite subsets
  $R_1$, $R_2$ and $R_3$ of the free monoids $S_1^\star$, $S_2^\star$ and
  $S_3^\star$ respectively, and maps $\iota_1$ and $\iota_2$ from $S_3^\star$
  to $S_1^\star$ and $S_2^\star$ respectively, such that $\Gamma$ admits an
  isomorphism to the group $\mod{\mathcal{P}}$ with presentation $\mathcal{P}$ of the form
\begin{align*}
  \langle S_1 \union S_2 \union S_3 \vert R_1 \union R_2 \union R_3 \union
    \{s^{-1}\iota_1(s), s^{-1}\iota_2(s) \colon s \in S_3\}\rangle
\end{align*}
where $\{\iota_1(r) \vert r \in R_3\} \subset R_1$ and $\{\iota_2(r) \vert r
\in R_3\} \subset R_2$. This ensures that $\iota_1$ and $\iota_2$ induce group
homomorphisms from $\langle S_3 \vert R_3\rangle$ to $\langle S_1 \vert
R_1\rangle$ and  $\langle S_2 \vert R_2\rangle$, which we shall denote
$\widehat\iota_1$ and $\widehat\iota_2$. Let $\widehat\iota$ be the induced map from
$\langle S_3 \vert R_3\rangle$ to $\Gamma$. Then the assumptions on the nature
of the splitting give the following conditions:
\begin{enumerate}
\item that $\langle S_3 \vert R_3 \rangle$ be virtually cyclic,
\item that $\widehat\iota_1$ and $\widehat\iota_2$ be injective,
\item that $\widehat\iota_1$ and $\widehat\iota_2$ not be surjective and
\item that for each $H \in \mathcal{H}$ with generating set $S_H$ there exists
  $g_H$ in $\Gamma$ such that either $S_1$ or $S_2$ contains the image of
  $g_HS_Hg_H^{-1}$ under the isomorphism from $\Gamma$ to $\mod{\mathcal{P}}$.
\end{enumerate}

We show that there is an algorithm that finds such a presentation for
$\Gamma$ if it exists. Using Tietze transformations, there is an algorithm that
enumerates all presentations of $\Gamma$ and for each presentation gives an
explicit isomorphism from $\Gamma$ to the realisation of that presentation;
therefore it is sufficient to show that each of the four conditions above can
be checked algorithmically.

The first condition can be checked using the first part of the algorithm of
Lemma~\ref{lem:twoendedsubgroups}.

$\widehat\iota_1$ and $\widehat\iota_2$ are both injective if and only if $\widehat\iota$
is injective: one direction is trivial, the other is a consequence of the
normal form theorem for the amalgamated product. To check this condition,
compute the maximal finite normal subgroup of $\langle S_3\vert R_3\rangle$ and
check the triviality of the image under $\widehat\iota$ of each of these elements.
Then find an element of $\langle S_3\vert R_3\rangle$ of infinite order and
check whether or not the image under $\widehat\iota$ of this element has infinite
order.

To check the surjectivity of $\widehat\iota_1$ first check whether or not $\langle
S_1 \rangle$ is virtually cyclic. If it is, check whether or not the maximal finite
normal subgroup of $\im\widehat\iota_1$ is equal to the maximal finite normal
subgroup of $\langle S_1 \rangle$. If it is, next check whether or not $\langle
S_1 \rangle$ and $\langle S_3 \rangle$ are either both of $\mathcal{Z}$-type or
both of $D_\infty$ type. If they are both of the same type, pass to an index 2
subgroup of each if necessary to ensure that they are both of $\mathcal{Z}$
type, then check whether or not the composition of $\widehat\iota_1$ with the
natural surjection to $\integers$ is surjective. If it is then $\widehat\iota_1$ is
surjective; if any of these tests produced the opposite answer then
$\widehat\iota_1$ is not surjective. Repeat this process for $\widehat\iota_2$.

The final condition can be checked using a solution to the simultaneous
conjugacy problem in $\Gamma$.

The existence of such a presentation for $\Gamma$ guarantees that $\Gamma$
splits non-trivially as an internal amalgamated product 
\begin{align*}
  \Gamma \isom \langle S_1 \vert R_1 \rangle \freeprod_{\langle S_3 \vert R_3
    \rangle} \langle S_2 \vert R_2 \rangle.
\end{align*}
This splitting is over a virtually cyclic
subgroup of $\Gamma$ and is relative to $\mathcal{H}$. 
\end{proof}

\subsection{Computing a maximal splitting}

We will need to be able to determine algorithmically whether or not the
boundary of the given hyperbolic group is homeomorphic to a circle. This is
achieved using the algorithm of Corollary~\ref{cor:noncutpaircomputable} and a
theorem from point-set topology, which we note here.

\begin{thm}\cite[II.2.13]{wilder49}\label{thm:topologyimpliescircle} Any
separable, connected, locally connected space containing more than one point
that is without a cut point and in which every pair is a cut pair is
homeomorphic to $S^1$.\end{thm}

\begin{prop}\label{prop:isthereasplitting} There is an algorithm that takes as
  input a presentation for a hyperbolic group $\Gamma$ with a collection
  $\mathcal{H}$ of virtually cyclic subgroups such that $\boundary(\Gamma,
  \widehat{\mathcal{H}})$ is connected and $\Gamma$ appears as the stabiliser
  of some vertex in the action of a hyperbolic group on a tree and
  $\mathcal{H}$ is a set of conjugacy class representatives of incident edge
  groups and returns the answer to the question ``does $\Gamma$ split
non-trivially over a virtually cyclic subgroup relative to
$\mathcal{H}$?''\end{prop}

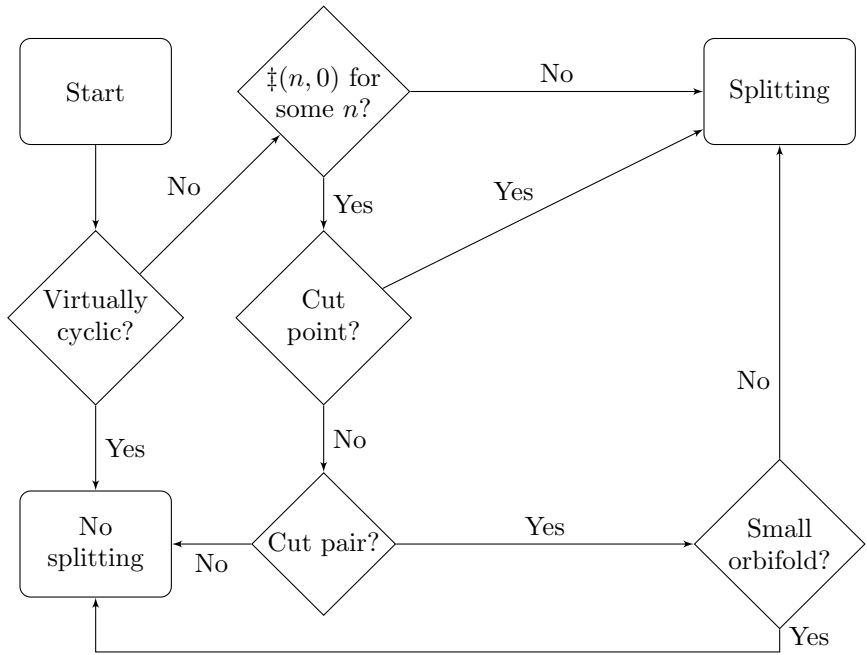
\begin{figure}
\centering
\begin{tikzpicture}[node distance = 2cm, auto]

\tikzstyle{decision} = [diamond, draw,
  text width=4.5em, text badly centered, node distance=3cm, inner sep=0pt]
\tikzstyle{startblock} = [rectangle, draw, 
  text width=5em, text centered, rounded corners, minimum height=4em]
\tikzstyle{acceptblock} = [rectangle, draw,
  text width=5em, text centered, rounded corners, minimum height=4em]
\tikzstyle{rejectblock} = [rectangle, draw,
  text width=5em, text centered, rounded corners, minimum height=4em]
\tikzstyle{line} = [draw, -latex']

\node [startblock] (start) {Start};
\node [decision, below of=start] (vcyclic) {Virtually cyclic?};
\node [decision, right of=start] (ddag) {$\ddag(n, 0)$ for some $n$?};
\node [decision, below of=ddag] (cutpoint) {Cut point?};
\node [acceptblock, right of=ddag, node distance=6cm] (splitting) {Splitting};
\node [decision, below of=cutpoint] (cutpair) {Cut pair?};
\node [decision, right of=cutpair, node distance=6cm] (small) {Small orbifold?};
\node [rejectblock, left of=cutpair, node distance=3cm] (nosplitting) {No splitting};

\path [line] (start) -- (vcyclic);
\path [line] (vcyclic) --  node {No} (ddag);
\path [line] (vcyclic) --  node {Yes} (nosplitting);
\path [line] (ddag) --  node {No} (splitting);
\path [line] (ddag) --  node {Yes} (cutpoint);
\path [line] (cutpoint) -- node {No} (cutpair);
\path [line] (cutpoint) -- node {Yes} (splitting);
\path [line] (cutpair) --  node {No} (nosplitting);
\path [line] (cutpair) --  node {Yes} (small);
\path [line] (small) --  node [near start] {No} (splitting);
\path [line] (small.south) -- ++(0,-.3cm) node[pos=.3] {Yes} -- ++(-9cm,0) -- (nosplitting);

\end{tikzpicture}

\caption{The decision process in the algorithm of Proposition~\ref{prop:isthereasplitting}.}
\end{figure}

\begin{proof} Let the given group be $\Gamma$ and the peripheral structure be
$\mathcal{H}$. First check whether or not $\Gamma$ is virtually cyclic. If it
is then $\Gamma$ does not split properly over a virtually cyclic subgroup.
If it is not then compute $\widehat{\mathcal{H}}$; then $\boundary(\Gamma,
\widehat{\mathcal{H}})$ contains more than a single point and the results
of this section can be applied.

Next compute $\delta$ such that the cusped space $X$ associated to the pair
$(\Gamma, \widehat{\mathcal{H}})$ is $\delta$-hyperbolic.  Search for a
non-trivial splitting of $\Gamma$ relative to $\mathcal{H}$ using
Lemma~\ref{lem:findingasplitting} and, in parallel, search for $n$ such that
$\ddag_n$ holds in $X$; one of these processes must terminate by
Proposition~\ref{prop:Xsatisfiesddag} and
Proposition~\ref{prop:cutpointimpliessplitting}.

If a splitting is found then $\Gamma$ does split non-trivially over a virtually
cyclic subgroup relative to $\mathcal{H}$, and the algorithm can return
``yes''. If $X$ satisfies $\ddag_n$ then use the algorithm of
Corollary~\ref{cor:cutpaircomputable} to check whether or not
$\boundary(\Gamma, \widehat{\mathcal{H}})$ contains a cut point. If it does then $\Gamma$
does split properly over a virtually cyclic subgroup by
Proposition~\ref{prop:cutpointimpliessplitting}

If there is no cut point, use the algorithm of
Corollary~\ref{cor:cutpaircomputable} on $X$ to determine whether or
$\boundary(\Gamma, \widehat{\mathcal{H}})$ contains a cut pair; if it does not then
$\Gamma$ does not split relative to $\mathcal{H}$ by
Proposition~\ref{prop:relativesplittingimpliescutpair}.

If there is a cut pair then simultaneously run the algorithms of
Lemma~\ref{lem:findingasplitting} and
Lemma~\ref{lem:identifyingsmallorbifolds}. If the former terminates then a
splitting has been found; if the latter does then no splitting exists.
\end{proof}

Note that a subprocess of this algorithm, together with the algorithm
of~\cite{gerasimov} that determines whether or not a hyperbolic group is
one-ended, provides the algorithm promised in
Theorem~\ref{thm:S1boundarycomputable}.

\begin{prop}\label{prop:maximalsplittingcomputable} There is an algorithm that,
when given a presentation for a hyperbolic group, computes the graph of groups
associated to a splitting of that group that is maximal for domination.
\end{prop}

\begin{proof} The algorithm iteratively constructs a sequence of
$\Gamma$-marked graphs of groups $G_i$. Let $G_1$ consist of a single vertex
with vertex group $\Gamma$. Then to obtain $G_{i+1}$ from $G_i$, use
Proposition~\ref{prop:isthereasplitting} to check whether each vertex group
splits non-trivially relative to its incident edge groups. If no vertex group
does split then halt the algorithm here. If the group at a vertex $v$ does
split then find the non-trivial graph of groups $G'$ with fundamental group
$\Gamma_v$ using Lemma~\ref{lem:findingasplitting}. Then define $G_{i+1}$ by
replacing the vertex $v$ of $G_i$ by the graph $G'$ and connecting edges
corresponding to the edges of $G_i$ incident at $v$ to $G'$ in the obvious
way.

This process must eventually stabilise: this follows from an accessibility
theorem~\cite{bestvinafeighn91} since the associated group actions on trees
constructed are minimal and reduced by construction. Let the corresponding
$\Gamma$-trees stabilise at a $\Gamma$-tree $T$. 

Suppose that another $\Gamma$-tree $T'$ with virtually cyclic edge stabilisers
dominates $T$. Let $v$ be a vertex of $T$. If $e$ is an incident edge then
$\Gamma_e$ is elliptic with respect to the action on $T'$, since it contains
the stabiliser of an edge of $T'$ as a subgroup of finite index.  Therefore
each incident edge subgroup of $\Gamma_v$ is elliptic with respect to $T'$. The
vertex stabiliser $\Gamma_v$ does not split over a virtually cyclic subgroup
relative to its incident edge groups, so $\Gamma_v$ is elliptic with respect to
$T'$. The vertex $v$ was arbitrary, so $T$ dominates $T'$, and so $T$ is maximal
for domination.  \end{proof}

\section{JSJ Decompositions}\label{sec:JSJcomputable}

In this section we show that three closely related types of JSJ splittings are
computable for hyperbolic groups. Fix a one-ended hyperbolic group $\Gamma$ and recall that 
$\mathcal{VC}$ is defined to be the set of all virtually cyclic subgroups of $\Gamma$,
$\mathcal{Z}$ is defined to be the set of all virtually cyclic subgroups of $\Gamma$
with infinite centre and $\mathcal{Z}_\text{max}$ is defined to be the set of
subgroups of $\Gamma$ in $\mathcal{Z}$ that are maximal for inclusion. We
consider the JSJ splittings over groups in these three sets.

\subsection{Two-ended edge groups}

We now prove the first part of Theorem~\ref{thm:maintheorem}.

\begin{thm}\label{thm:VC-JSJcomputable} There is an algorithm that takes as
input a presentation for a one-ended hyperbolic group and returns as output the graph
of groups associated to a $\mathcal{VC}$-JSJ decomposition for that group.
This decomposition can be taken to be Bowditch's canonical
decomposition.\end{thm}

We first prove the following lemma.

\begin{lem}\label{lem:JSJfrommaximal} The tree obtained from the tree
associated to a maximal splitting by collapsing each edge whose stabiliser is
not universally elliptic is a $\mathcal{VC}$-JSJ tree.\end{lem}

\begin{proof} Let $T$ be the tree associated to a maximal splitting and let
$T'$ be the tree obtained by collapsing each edge of $T$ that is not
universally elliptic. Then certainly $T'$ is universally elliptic, so it is
sufficient to show that if $\Sigma$ is another universally elliptic
$\Gamma$-tree then $T'$ dominates $\Sigma$.

The tree $\Sigma$ can be refined to dominate $T$, so there exists a map $f
\colon T \to \Sigma$.  Let $v$ be a vertex of $T'$ and let $S$ be a
component of its preimage in $T$.  Then $f\restricted{S}$ is constant: if
an edge $e$ in $S$ is mapped into an edge $e'$ in $\Sigma$ then $\Gamma_e \leq
\Gamma_{e'}$, which is universally elliptic.  But then the image of $S$ in
$T'$ contains more than a single vertex.  Therefore $\Gamma_v$ fixes the
vertex $f(S)$ in $\Sigma$, so is elliptic with respect to $\Sigma$. This shows that
$T'$ dominates $\Sigma$.\end{proof}

We must now identify the edges in the tree associated to the maximal splitting
that are not maximally elliptic. We make the following definitions.

\begin{defn}\label{defn:mobiusstripgroup} An \emph{extended M\"obius strip
group} is a virtually cyclic group of $\mathcal{Z}$ type with peripheral
structure consisting of a single index 2 subgroup.\end{defn}

\begin{defn}\label{defn:surfaceboundaryedges} We say that an edge $e$
connecting vertices $v_1$ and $v_2$ of a $\Gamma$-tree is a \emph{internal
surface edge} if, for each $i$, either $\Gamma_{v_i}$ is a hanging Fuchsian
group and $\Gamma_e$ is maximal among virtually cyclic subgroups of
$\Gamma_{v_i}$, or $\Gamma_{v_i}$ is an extended M\"obius strip group and
$\Gamma_e \leq \Gamma_{v_i}$ is the peripheral subgroup of $\Gamma_{v_i}$.
\end{defn}

\begin{lem}\label{lem:nonuniversallyellipticedges} If $T$ is reduced (that
is, no proper collapse of $T$ dominates $T$) then the edges of
$T$ that are not universally elliptic are precisely the internal surface
edges.\end{lem}

\begin{proof} Let $T'$ be the tree obtained by collapsing each edge of
$T$ that is not universally elliptic as in
Lemma~\ref{lem:JSJfrommaximal}, so $T'$ is a JSJ tree and there is a
collapse map from $T$ to $T'$. The edges of $T$ that are not
universally elliptic are precisely those edges that are mapped to flexible
vertices of $T'$ under the collapse map; by~\cite[Theorem
6.2]{guirardellevitt16} all flexible vertices of $T'$ are hanging
Fuchsian vertices. 

Any splitting of a hanging Fuchsian group is dual to a family of curves on the
associated orbifold, so any edge in such a splitting is an internal surface
edge, so all edges that are not universally elliptic are internal surface
edges.

Conversely, let $e$ be an internal surface edge. Let $T'$ be the tree
obtained by collapsing each edge in the orbit of $e$; let $v$ be the vertex of
$T'$ in the image of $e$. Then $T$ is obtained from $T'$ by
refining at $v$. The tree $T$ was assumed to be reduced and $v$ is a
hanging Fuchsian vertex so this refinement is dual to an essential simple
closed curve $\ell$ on the associated orbifold $Q$. Then $Q$ contains another
essential simple closed curve $\ell'$ that is not homotopic to a curve disjoint
from $\ell$. Refine $T'$ at $v$ dual to $\ell'$ to obtain a tree
$T''$; then $\Gamma_e$ is not elliptic with respect to $T''$.
\end{proof}

We now have sufficient tools to prove the computability of a $\mathcal{VC}$-JSJ
for a given hyperbolic group $G$. In~\cite{bowditch98}, Bowditch defines a
canonical JSJ in the class of all $\mathcal{VC}$-JSJs of a given hyperbolic
group. In the language of~\cite{guirardellevitt16} this is the decomposition
corresponding to the \emph{tree of cylinders} of any other $\mathcal{VC}$-JSJ.

\begin{defn} Let $T$ be a $\mathcal{VC}$-tree. Define the \emph{commensurability}
equivalence relation $\sim$ on $\mathcal{VC}$ by letting $A \sim B$ if
and only if $A$ and $B$ lie in the same maximal virtually cyclic subgroup of
$\Gamma$. Also denote by $\sim$ the equivalence relation on the set of
edges of $T$ defined by letting $e \sim e'$ if and only if $\Gamma_e \sim
\Gamma_{e'}$. A \emph{cylinder} is a subset $Y\subset T$ that is the union of
all edges in a $\sim$-equivalence class.\end{defn}

\begin{defn} Let $T$ be a $\mathcal{VC}$-tree. The corresponding \emph{tree of
  cylinders} $T_c$ is a bipartite tree with vertex set $V_1 \disjointunion
  V_2$, where $V_1$ is the set of vertices of $T$ that lie in at least two
  cylinders and $V_2$ is the set of cylinders in $T$. A vertex $v \in V_1$ is
connected by an edge to $Y \in V_2$ if and only if $v \in Y$.\end{defn}

\begin{proof}[Proof of Theorem~\ref{thm:VC-JSJcomputable}] First compute a
maximal splitting of the group over virtually cyclic subgroups by
Theorem~\ref{prop:maximalsplittingcomputable}. Let $T$ be the associated
tree. By construction $T$ is reduced; in any case, $T$ can easily be made
reduced using the processes of Lemma~\ref{lem:twoendedsubgroups}. For each
edge $e$ connecting vertices $v_1$ and $v_2$ of the graph of groups
$T/\Gamma$ determine whether $\Gamma_e$ is maximal in $\Gamma$ using the
algorithm of Lemma~\ref{lem:twoendedsubgroups} and whether the two vertex
groups $\Gamma_{v_1}$ and $\Gamma_{v_2}$ have circular boundary relative to
their incident edge groups by Theorem~\ref{thm:S1boundarycomputable}. Check
also whether each of $\Gamma_{v_1}$ and $\Gamma_{v_2}$ is virtually cyclic of
$\mathcal{Z}$-type, and, if it is, whether or not $\Gamma_e$ has index 2{} in
that group. One of these possibilities is the case if and only if $e$ is not
universally elliptic by Lemma~\ref{lem:nonuniversallyellipticedges}; collapse
all edges where this is the case.

Bowditch's canonical decomposition is the graph of cylinders of the
decomposition obtained in this way. The operation of replacing a decomposition
with the decomposition associated to its tree of cylinders can be done
algorithmically using~\ref{lem:twoendedsubgroups}. This is result the content
of~\cite[Lemma 2.34]{dahmaniguirardel11}; note that while the result is stated
for a $\mathcal{Z}$-tree, replacing this with a $\mathcal{VC}$-tree makes no
difference to the proof.\end{proof}

\subsection{\texorpdfstring{$\mathcal{Z}$}{Z} edge groups}

We now prove the second part of Theorem~\ref{thm:maintheorem}.

\begin{thm} There is an algorithm that takes as input a presentation for a
one-ended hyperbolic group and returns the graph of groups associated to a
$\mathcal{Z}$-JSJ decomposition for that group.\end{thm}

In~\cite{dahmaniguirardel11} it is shown that the $\mathcal{Z}$-JSJ
decomposition is closely related to the $\mathcal{VC}$-JSJ: a $\mathcal{Z}$-JSJ
tree can be obtained from a $\mathcal{VC}$-JSJ tree $T$ by first refining $T$
by applying the so-called mirrors splitting to each hanging Fuchsian vertex group and
then collapsing each edge with stabiliser of dihedral type. 
The second of these processes can be done algorithmically using the part of the
algorithm of Lemma~\ref{lem:twoendedsubgroups} that determines whether or not a
given virtually cyclic group is of dihedral type.  Therefore we must now show
that the mirrors splitting is computable.

Recall the definition of the mirrors splitting of the fundamental group of a
compact 2-dimensional orbifold $Q$ from~\cite{dahmaniguirardel11}. 

\begin{defn} Let $N$ be
a regular neighbourhood of the union of the mirrors and $D_\infty$-boundary
components of $Q$ that does not contain any cone point of $Q$. If $Q - N$ is an
annulus or a disc with at most one cone point then the \emph{mirrors splitting} of
$\pi_1Q$ is defined to be trivial; otherwise it is the splitting obtained by
cutting $Q$ along each component of $\boundary N$. \end{defn}

If the mirrors splitting is non-trivial then the graph of groups associated to the splitting is a star; the group at the central
vertex is the fundamental group of an orbifold with no mirrors and the group at
each leaf is the fundamental group of an orbifold with no cone points and
underlying surface an annulus, one of whose topological boundary components is
a circular orbifold boundary component and the other a union of interval
boundary components and at least one mirror. If $\Gamma$ is any hyperbolic
group with a collection $\mathcal{H}$ of virtually cyclic subgroups such that
$\boundary(\Gamma, \widehat{\mathcal{H}})$ is homeomorphic to a circle then the mirrors
splitting of $\Gamma$ relative to $\mathcal{H}$ is defined to be the splitting
induced by the mirrors splitting of the quotient $\Gamma$ by a maximal finite
normal subgroup of $\Gamma$ as in Proposition~\ref{lem:grouptoorbifold}.

\begin{lem} The mirrors splitting of a hyperbolic group $\Gamma$ with a set
$\mathcal{H}$ of virtually cyclic subgroups such that $\boundary(\Gamma,
\widehat{\mathcal{H}})$ is homeomorphic to a circle is computable.\end{lem}

\begin{proof} Using Proposition~\ref{lem:grouptoorbifold} it is enough to show
  that the mirrors splitting is computable in the case where $\Gamma$ is
  bounded Fuchsian and $\mathcal{H}$ is the a collection of representatives of
  peripheral subgroups of $\Gamma$. To do this we enumerate all mirrors
  splittings: for each non-negative integer $k$ enumerate all fundamental
  groups of compact orbifolds without mirrors and with at least $k$ boundary
  components and all $k$-tuples of fundamental groups of orbifolds homeomorphic
  to an annulus with no cone points and such that one topological boundary
  component of the orbifold is a circular orbifold boundary component. In each
  case form the graph of groups in which the underlying graph is a $k$-pointed
  star, the group at the central vertex is the fundamental group of the
  orbifold without mirrors, the group at each leaf is the fundamental group of
  an orbifold homeomorphic to an annulus and the group at each edge is infinite
  cyclic and is identified with the fundamental group of a circular orbifold
  boundary component of each of the orbifolds associated to the end points of
  that edge. Compute the fundamental group of each such graph of groups and
  record also the peripheral structure consisting of conjugacy class
  representatives of the fundamental groups of components of the orbifold
  boundary of the orbifold.

  Also enumerate all groups with trivial mirrors splitting, i.e.\ fundamental
  groups of orbifolds homeomorphic as topological spaces to a disc with at most
  one cone point, or homeomorphic to an annulus with no cone points and such
  that one topological boundary component is a circular orbifold boundary
  component.

  In parallel enumerate all homomorphisms from the fundamental groups of these
  graphs of groups to $\Gamma$ and all homomorphisms from $\Gamma$ to the
  fundamental groups of these graphs of groups. Some such pair of homomorphisms
  is an inverse pair that preserves the peripheral structure up to conjugacy.
  On finding this pair the algorithm returns the associated mirrors splitting.
\end{proof}

\subsection{\texorpdfstring{$\mathcal{Z}_\text{max}$}{Zmax} edge groups}

In~\cite{dahmaniguirardel11} it is shown that the $\mathcal{Z}_\text{max}$-JSJ
decomposition can be obtained from a $\mathcal{Z}$-JSJ decomposition by
performing the so-called $\mathcal{Z}_\text{max}$-fold. We note that this can
be done algorithmically, which completes the proof of the final part of
Theorem~\ref{thm:maintheorem}, which we restate here as Theorem~\ref{thm:ZmaxJSJ}.

\begin{thm}\label{thm:ZmaxJSJ} There is an algorithm that takes as input a presentation for a
one-ended hyperbolic group and outputs the graph of groups associated to a
$\mathcal{Z}_\text{max}$-JSJ decomposition for that group.\end{thm}

\begin{lem} There is an algorithm that takes as input a graph of groups
decomposition of a hyperbolic group $\Gamma$ over virtually cyclic subgroups
and returns the graph of groups associated to the
$\mathcal{Z}_\text{max}$-fold of the associated $\Gamma$-tree.\end{lem}

\begin{proof} This can be done by iterating some simple folds described
in~\cite{dahmaniguirardel11}. We describe this fold at the level of the tree
$T$. Take some edge $e$ in $T$ such that the maximal subgroup $\hat\Gamma_e$
of $\Gamma$ containing $\Gamma_e$ is a subgroup of $\Gamma_{o(e)}$ where $o(e)$ is
the origin of $e$. Then take the quotient of $T$ by the $\Gamma$-equivariant
equivalence relation generated by $e \sim \hat\Gamma_e \cdot e$. Note that this
does not change the underlying graph of the associated graph of groups.

At the level of the graph of groups this fold is achieved by taking an edge $e$
in the graph such that $\hat \Gamma_e$ is a subgroup of $\Gamma_{o(e)}$,
replacing the group at the edge $e$ by $\hat\Gamma_e$ and replacing the group
at the terminal vertex $t(e)$ of $e$ by $\langle\hat\Gamma_e,
\Gamma_{t(e)}\rangle$. Repeat this process until there is no edge $e$ such that
$\Gamma_e \neq \hat\Gamma_e$. \end{proof}

\end{document}